\documentclass[11pt]{amsart}
\usepackage[margin=1in]{geometry}
\usepackage{amsmath,amssymb,amsthm}
\usepackage{mathtools}
\usepackage{enumitem}
\usepackage{booktabs}
\usepackage{hyperref}
\usepackage{xcolor}
\usepackage{bm, comment}

\theoremstyle{plain}
\newtheorem{theorem}{Theorem}[section]
\newtheorem{proposition}[theorem]{Proposition}
\newtheorem{lemma}[theorem]{Lemma}
\newtheorem{corollary}[theorem]{Corollary}

\theoremstyle{definition}

\newtheorem{example}[theorem]{Example}
\newtheorem{remark}[theorem]{Remark}
\theoremstyle{plain}
\newtheorem*{theoremA}{Main Theorem}

\DeclareMathOperator{\Res}{Res}

\newcommand{\ZZ}{\mathbb Z}
\newcommand{\QQ}{\mathbb Q}
\newcommand{\CC}{\mathbb C}
\newcommand{\SU}{\mathrm{SU}}

\title[Resultants, Recursive Formulas and Binet Formulas for $\textrm{\rm SU}(N)$ Verlinde Sums]
{Resultants, Recursive Formulas and Binet Formulas for $SU(N)$ Verlinde Sums}

\author{Jay Jorgenson}\thanks{The first named author acknowledges grant support from PSC-CUNY Awards 67415-00 55 and 68462-00 56, which are jointly funded by the Professional Staff Congress and The City University of New York.}
\author{Anders Karlsson}\thanks{The second-named author was supported by the Swiss NSF Grants 200020-200400 and 200021-212864, and by the Swedish Research Council Grant 104651320.}
\author{Lejla Smajlovi\'c}

\date{}

\begin{document}

\begin{abstract}
We study Verlinde sums $V_{n}(N,m)$, which for our purposes are finite sums of trigonometric functions $\{r_{x}\}$,
associated to $\textrm{\rm SU}(N)$.   In this article we prove the following results:
(i) the polynomial $D_N(w):=\prod_x(w-r_x)$ is equal to, up to an explicit and computable factor, the iterated
resultant of $Q(u)=u^m-1$ with an explicit elementary polynomial $W_N$ which is independent of $n$ and $m$;
(ii) for fixed $N$ and $m$, $V_n(N,m)$ satisfies a linear recursion of order at most
$\binom{m-1}{N-1}$, together with arguments showing that one actually has a recursion of length equal to
the number of distinct roots of $D_{N}$; (iii) for fixed $N$ and $m$, a Binet-type closed form expressing $V_n(N,m)$ as an explicit
finite sum of $n$-th powers of algebraic numbers which are rescaled reciprocals of the distinct roots of $D_N(w)$ is derived. Several explicit examples are provided for $N=2,3,4,5,6$ and $8$.
\end{abstract}

\maketitle

\section{Introduction}\label{sec:intro}

\subsection{The object of study}

Verlinde's formula, originally proposed in conformal field theory, computes the dimension of a finite-dimensional vector space associated with a compact Lie group, a Riemann surface of genus $g$, and a positive integral level $k$~\cite{Ve88}. Its mathematical interpretation in terms of conformal blocks and generalized theta functions, together with rigorous proofs and algebro-geometric formulations, was developed in \cite{BL94,TUY90,Fa94,Be96}; see also the monograph~\cite{Ku21}.

For integers $N\ge2$, $m>N$, and $n\ge1$, set $x_0:=0$ and define
\begin{equation}\label{eq:VSum}
V_n(N,m) := 2\bigl(Nm^{N-1}\bigr)^n \sum_{0<x_1<\cdots<x_{N-1}<m} \prod_{0\le a<b\le N-1} \left(2\sin\frac{\pi(x_b-x_a)}{m}\right)^{-2n}.
\end{equation}
We call $V_n(N,m)$ a \emph{Verlinde sum}. With the normalization used in~\cite{Za96}, it is related to the $\SU(N)$ Verlinde dimension for genus $g=n+1$ at level $k=m-N$ by
\begin{equation*} V_n(N,m)=2D(n+1,N,m), \qquad D(g,N,m):= \dim H^{0}\!\bigl( \mathcal N_{g,N,0},\Theta^{\otimes(m-N)} \bigr),
\end{equation*}
where $\mathcal N_{g,N,0}$ is the moduli space of semistable rank $N$ vector bundles with trivial determinant over a fixed Riemann surface of genus $g$, and $\Theta$ denotes the determinant line bundle; see~\cite{Za96}. 

Following the point of view adopted in \cite{Za96} and \cite{JKS26}, we suppress the geometric, algebraic, topological and physical background attached to $V_n(N,m)$, and study directly the finite
trigonometric expression on the right-hand side of~\eqref{eq:VSum}, with particular emphasis on their spectral interpretation and effective computation.  As noted in \cite{Za96}, starting from $N\geq 4$, the direct calculation of these sums is not so straightforward, \emph{cf.} Example \ref{subsec:N4} below. 

The combinatorial study of such sums began with the verification and analysis of the
$\SU(2)$ and $\SU(3)$  cases in~\cite{Do92,Sz91,Sz93,Za96} and was developed through iterated residues,
Bernoulli series, vector partition functions, and Euler--Maclaurin formulas
in~\cite{BoVe09,BrVe97,Sz98,Sz03,SV03}. For example, in \cite{Sz03} a general residue theorem for rational trigonometric sums was derived and Verlinde sums for a fixed genus and varying level were expressed as a finite sum of iterated residues determined by the underlying lattice and hyperplane arrangement.
Verlinde sums were also treated as rational trigonometric sums in~\cite{LM22}, where a decomposition formula for verlinde sums, similar to the Boysal-Vergne decomposition formula \cite{BV12} for Bernoulli series was established. Spectral and discrete-analytic methods
for related trigonometric sums appear in~\cite{CHJSV23a,Do92,JKS24}.  

These approaches primarily address the dependence on the level for fixed genus, or explicit rank-two summation formulas. The present work instead develops an algebraic combinatorics method for evaluation of Verlinde sums in terms of the genus parameter at fixed level.
More precisely, the present article extends the combinatorial, algebraic and algorithmic
aspects developed for $\SU(3)$ in~\cite{JKS26}, focusing exclusively on certain
algebraic and combinatorial features of $V_n(N,m)$. In particular, we
establish the ordinary generating function in $n$ for fixed $m$ and $N$,
formulate linear recursion formulas in $n$, and develop Binet-type formulas for
$V_n(N,m)$.

\subsection{Our results}

In the present article we focus on generalizing the combinatorial, algebraic and algorithmic aspects from \cite{JKS26} to $\text{\rm SU}(N)$ Verlinde sums.  For $u = (u_0, u_1, \dots, u_{N-1})$, let us write
\begin{equation}\label{eq. Delta}
\Delta(u) := \prod_{0 \le a < b \le N-1} (u_a - u_b).
\end{equation}
For notational convenience, we adopt the convention that $u_0 := 1$. Define the polynomial
\begin{equation}\label{eq:W_definition}
W_N(u_1, \dots, u_{N-1}; w) \;:=\; w \prod_{a=1}^{N-1} u_a^{\,N-1} \;-\; (-1)^{\binom{N}{2}} \,\Delta(1, u_1, \dots, u_{N-1})^2.
\end{equation}
For integers $m>N>0$ define the set
\begin{equation}\label{eq:alcove}
\mathcal A(m,N):=\big\{x=(x_1,\dots,x_{N-1})\in\ZZ^{N-1} : 0<x_1<\cdots<x_{N-1}<m\big\},
\end{equation}
which we will refer to as the \emph{alcove}. In essence, $\mathcal A(m,N)$ is an analogue of the circulant graph
from \cite{JKS26} but modeled on the alcove of the complex Lie algebra $A_{N-1}$ which is associated to $\textrm{\rm SU}(N)$,
but that point of view plays no role in our considerations.  For each $x \in \mathcal A(m,N)$ we define a number $r_{x}$ which is a product
of values of trigonometric functions; see \eqref{eq:r_definition}.    In order to state our results, we need that the cardinality $M_N(m)$ of $\mathcal A(m,N)$ equals $\binom{m-1}{N-1}$, and we set $c_{N,m} := Nm^{N-1}$.

Define
\begin{equation}
r_x :=\; \prod_{0\le a<b\le N-1} 4\sin^2\tfrac{\pi(x_b-x_a)}{m},
\end{equation}
and
\begin{equation}\label{eq. DN}
D_N(w) := \prod_{x \in \mathcal A(m,N)} (w - r_x).
\end{equation}

\begin{theoremA}\label{thm:main}
Fix $N \ge 2$ and $m > N$. We have the following results.
\begin{enumerate}
\item[\textup{(i)}] \textbf{(Resultant formula)} Let $Q(u) := u^m - 1$. Then
\begin{equation}\label{eq:resultant}
D_N(w)^{(N-1)!} = (f(w))^{-1}\operatorname{Res}_{u_1}\!\Big(Q(u_1), \dots, \operatorname{Res}_{u_{N-1}}\big(Q(u_{N-1}), W_N(u_1,\dots,u_{N-1};w)\big)\dots\Big)
\end{equation}
where
$$
f(w) = \varepsilon_{N,m}\, w^{\,m^{N-1} - M_N(m)(N-1)!}
$$
and
$$
\varepsilon_{N,m} = 1
\,\,\,\,
\textrm{\rm for $N \ge 3$ and every $m$, and $\varepsilon_{2,m} = (-1)^{m-1}$.}
$$
In particular $D_N(w)$ is obtained by computing the $(N-1)!$-th root of \eqref{eq:resultant}.
\item[\textup{(ii)}] \textbf{(Generating function)} For generating function purposes, set $V_0(N,m):=2M_N(m)$. With $G_N(w) := D_N'(w)/D_N(w)$,
$$
\sum_{n \ge 0} V_n(N,m)\,\tau^n \;=\; 2\Big(M_N(m) - v\,G_N(v)\Big)\Big|_{v = c_{N,m}\tau},
$$
as an identity of formal power series in $\tau$. Analytically, both sides of the above equation agree and the series on the left hand side converges uniformly and absolutely for $|\tau|\leq R<\min_xr_x/c_{N,m}$.

\item[\textup{(iii)}] \textbf{(Fixed $m$ and $N$ linear recursion)} Writing $D_N(c_{N,m}\tau) = \sum_k q_k \tau^k$, the sequence $\big(V_n(N,m)\big)_{n \ge 0}$ satisfies the linear recursion with constant coefficients
$$
\sum_{k=0}^{\min(n,M_N(m))} q_k\, V_{n-k}(N,m) \;=\; \begin{cases} p_n, & 0 \le n \le M_N(m), \\ 0, & n > M_N(m), \end{cases}
$$
of  order at most $M_N(m)$, where $p_n$ is the $n$-th coefficient of $P_N(\tau) := 2\big(M_N(m) D_N(c_{N,m}\tau) - c_{N,m}\tau D_N'(c_{N,m}\tau)\big)$.

\item[\textup{(iv)}] \textbf{(Binet formula)} Let $\{a\}$ denote the distinct values among $\{r_x\}_{x \in \mathcal A(m,N)}$, with $\mu(a)$ being the cardinality of the set $\{x\in\mathcal{A}(m,N):\, r_x=a\}$. Set $y_a := c_{N,m}/a$. Then for every $n \ge 1$,
$$
V_n(N,m) \;=\; 2\sum_{a} \mu(a)\, y_a^{\,n}.
$$
\end{enumerate}
\end{theoremA}

Part (iii) gives an explicit, finite recursion for $V_n(N,m)$ in $n$, for fixed $N$ and $m$, yet of order at most $M_N(m)$, which grows quickly with $m$.
 We continue our study and show that the order of recursion in (iii) can always be shortened to be at most $K_{\min}(N,m)$, which can
be bounded \it a priori \rm using two different symmetry considerations: the Galois action on $\mathbb{Q}(\zeta_m)^+$, which bounds
the degree of each irreducible factor of $D_N$, and a dihedral symmetry, which explains many coincidences among roots and yields an upper bound for the number of distinct values. The characteristic roots of this shortened recursion are exactly
the $y_a$ of part (iv), giving a closed form for $V_n(N,m)$ that directly generalizes the classical Binet formula \begin{equation}\label{eq:Classical_Binet}
F_n \;=\; \frac{1}{\sqrt5}\left[\left(\frac{1+\sqrt5}{2}\right)^{\!n} - \left(\frac{1-\sqrt5}{2}\right)^{\!n}\right]
\,\,\,\,\,
\textrm{\rm $n \ge 0$}
\end{equation} 
for the $n$-th Fibonacci number $F_{n}$.

For example, when $N=5$ and $m=10$, in Section 6 below, we prove that $V_n(5,10)$ can be expressed in terms of the Lucas numbers $L_k$ as
\begin{equation}\label{Vn5,10}
V_n(5,10)=2\Bigl[\,5\cdot400^{n}+10\cdot125^{n}+16^{n}
+5\cdot2000^{n}\,L_{6n}
+10\cdot500^{n}\,L_{4n}
+\bigl(20\cdot500^{n}+10\cdot125^{n}+10\cdot100^{n}\bigr)L_{2n}\Bigr],
\end{equation}
while
\begin{equation}\label{Vn8,10}
V_n(8,10)=8\cdot20^{\,n}+16\cdot80^{\,n}\,L_{2n}+16\,s_n,
\qquad
s_n=16^{\,n}\,5^{\lceil n/2\rceil}
\begin{cases}L_n,&n\text{ even}\\[2pt]F_n,&n\text{ odd}.\end{cases}
\end{equation}

One way by which our results can be summarized is through the following algorithm, which
references results from the text below.

\medskip
\textbf{The algorithm}\label{sec:algorithm}
\begin{quote}
\textbf{Input:} Integers $N\ge2$, $m > N$, $n\ge1$.
\begin{enumerate}[label=(\arabic*)]
\item Let $Q(u):=u^m-1$ and $W_N$ as in \eqref{eq:W_definition}.
\item Compute the iterated resultant of Theorem~\ref{thm:resultant} and extract $D_N(w)$ via Corollary~\ref{cor:extract} (an exact $(N-1)!$-th root, with the power of $w$ known in advance, so no numerical root-extraction ambiguity arises).
\item Factor $D_N(w)$ over $\QQ$ and read off the distinct roots $a$ and their multiplicities $\mu(a)$ (Theorem~\ref{thm:binet}).
\item Set $y_a:=Nm^{N-1}/a$.
\end{enumerate}
\textbf{Output:} $V_n(N,m) = 2\sum_a \mu(a)\,y_a^{\,n}$ (Theorem~\ref{thm:binet}), or, for a sequence of $n$'s, the order-$K_{\min}(N,m)$ recursion of Theorem~\ref{thm:minrec} with seeds $V_1(N,m),\dots,V_{K_{\min}(N,m)}(N,m)$ computed once from the Binet formula.
\end{quote}

\subsection{Organization}
In Section 2 we establish further information and key results for our study. In Section 3 we prove the first two parts of our main theorem.  We build
the auxiliary polynomial $W_N$ from a single elementary fact about complex conjugation on the unit
circle, and in Section 3.2 we give a complete, self-contained proof of the resultant identity for $D_N(w)$.
In Section 3.3 we compare this
construction at $N=3$ with the resultant construction of \cite{JKS26}, giving an exact formula
relating the two resultants. In Section 4 we prove the (unreduced) fixed $m$ recursion as
well as the Binet formula, thus completing the proof of the Main Theorem.  As stated
above, this recursion formula is prohibitively long, so we next investigate reductions of the length, which amounts to studying the
multiplicities of the factors of $D_{N}(w)$.  Specifically, in Section 5 we reduce the recursion using the field-theoretic structure of $\{r_x\}$.  In Section 5.1 we obtain a bound
from Galois theory, and in Section 5.2 we get a further dihedral bracelet bound,
which accounts for the bulk of the reduction.  In Section 5.3 we identify the true minimal
recursion order $K_{\min}(N,m)$.  Finally, in Section 6
we present many fully worked, numerically verified examples for $N = 2,3,4,5,6$ and $8$. We highlight
the example $N=8$, $m=10$ as a conclusion.

\subsection{Comments on the use of AI}
We used AI tools for the proof-reading and copy-editing of our manuscript. In addition, they were used for symbolic calculations of polynomials. On the other hand, all proofs and methodology were developed by the authors.

\section{Background}\label{sec:background}

\subsection{The resultant}

Let $R$ be an integral domain contained in an algebraically closed field $\overline R$, and let
\[
f(u) = \sum_{i=0}^{m} f_i u^i, \qquad g(u) = \sum_{j=0}^{n} g_j u^j \qquad (f_m\ne0,\ g_n\ne0)
\]
be polynomials in $R[u]$ of degrees $m,n\ge0$. The \emph{Sylvester matrix} $\mathrm{Syl}(f,g)$ is the $(m+n)\times(m+n)$ matrix whose first $n$ rows are successive shifts of the coefficient vector of $f$ and whose last $m$ rows are successive shifts of the coefficient vector of $g$, meaning
\[
\mathrm{Syl}(f,g) = \begin{pmatrix}
f_m & f_{m-1} & \cdots & f_0 & & & \\
 & f_m & f_{m-1} & \cdots & f_0 & & \\
 & & \ddots & & & \ddots & \\
 & & & f_m & f_{m-1} & \cdots & f_0 \\
g_n & g_{n-1} & \cdots & g_0 & & & \\
 & g_n & g_{n-1} & \cdots & g_0 & & \\
 & & \ddots & & & \ddots & \\
 & & & g_n & g_{n-1} & \cdots & g_0
\end{pmatrix}.
\]
The \emph{resultant} of $f$ and $g$ is the determinant of $\mathrm{Syl}(f,g)$, written as $\Res(f,g) := \det \mathrm{Syl}(f,g) \in R$.
Note that if the coefficients of $f$ or $g$ depend on a variable, say $w$, then so does $\Res(f,g)$.

\begin{lemma}\label{lem:resultant-basics}
Assume notation as above.
\begin{enumerate}[label=(\roman*)]
\item The resultant $\Res(f,g)=0$ if and only if $f$ and $g$ have a common root in $\overline R$.
\item If $f$ is monic with roots $\alpha_1,\dots,\alpha_m\in\overline R$ counted with multiplicity, then
\[
\Res(f,g) = \prod_{i=1}^{m} g(\alpha_i).
\]
\item If $f=f^{(1)}f^{(2)}$, then $\Res(f,g)=\Res(f^{(1)},g)\,\Res(f^{(2)},g)$.
\end{enumerate}
\end{lemma}

These facts are entirely classical; see any standard treatment of elimination theory.

\begin{remark}\label{rem:iterated}
Let $f_1(u_1),\dots,f_k(u_k)$ be monic polynomials and let $g(u_1,\dots,u_k,w)$ be any polynomial.
By iteratively applying Lemma~\ref{lem:resultant-basics}(ii), working from the innermost variable outward, we get that
\begin{equation}\label{eq:bigresultant}
\Res_{u_1}\Big(f_1(u_1),\dots,\Res_{u_k}\big(f_k(u_k),\,g(u_1,\dots,u_k,w)\big)\dots\Big) \;=\; \prod_{\substack{u_1:\,f_1(u_1)=0 \\ \ \ \vdots \\ u_k:\,f_k(u_k)=0}} g(u_1,\dots,u_k,w).
\end{equation}
The result \eqref{eq:bigresultant} is a polynomial in $w$ alone which is equal to the product of $g$ over the entire
set of roots of $f_1,\dots,f_k$.  If $d_i=\deg f_i$ and $g$ has degree $d$ in $w$, then the right-hand side of \eqref{eq:bigresultant} contains $\prod_{i=1}^k d_i$ factors and has degree at most $d\prod_{i=1}^k d_i$ in $w$.  Equality holds when the leading coefficient of $g$ in $w$ does not vanish at any root tuple.
\end{remark}

\subsection{The alcove}\label{sec:vandermonde}

For $u=(u_0,\dots,u_{N-1})$, recall \eqref{eq. Delta} and for $m > N$, recall the definition \eqref{eq:alcove} of the alcove.
One can show that $|\mathcal A(m,N)|=\binom{m-1}{N-1}$. For notational convenience, we set $x_0:=0$ throughout. For $x\in\mathcal A(m,N)$, let
$$
u_a:=e^{2\pi i x_a/m}
\,\,\,\,\,
\textrm{\rm for}
\,\,\,\,\,
a=0,\dots,N-1,
$$
so $u_0=1$.  Define
\begin{equation}\label{eq:r_definition}
r_x := |\Delta(u(x))|^2 \;=\; \prod_{0\le a<b\le N-1} 4\sin^2\tfrac{\pi(x_b-x_a)}{m}.
\end{equation}
In this notation, the $\textrm{\rm SU}(N)$ Verlinde sum at level $k=m-N$, genus $n+1$, is given by
\begin{equation}
V_n(N,m) \;:=\; 2\big(Nm^{N-1}\big)^{n}\sum_{x\in\mathcal A(m,N)} r_x^{-n}. \label{eq:defn}
\end{equation}
Given that $N$ is fixed, we will, in a slight abuse of notation, call $m$ the level.

\begin{lemma}\label{lem:vandermonde}
For every $x\in\mathcal A(m,N)$, $r_x=|\Delta(u(x))|^2$ is a real, strictly positive, algebraic integer.
\end{lemma}

\begin{proof}
Note that $r_x=\prod_{a<b}|u_a-u_b|^2$ is a finite product of squared moduli of differences of points on the unit circle,
hence real and nonnegative; it is strictly positive since $0=x_0<x_1<\cdots<x_{N-1}<m$ forces the $u_a$ pairwise distinct.
Each factor $|u_a-u_b|^2=2-u_a\bar u_b-\bar u_a u_b$ is an algebraic integer in $\ZZ[\zeta_m]$, so $r_x$ is an algebraic integer.
\end{proof}

Observe that complex conjugation acts on $\{u_a(x)\}$ by $x\mapsto -x\pmod m$ (followed by re-ordering in the alcove) and fixes $r_x$, hence $|\Delta(u(x))|^2 = |\Delta(u(-x))|^2$ and $r_x\in\QQ(\zeta_m)^+$, the maximal real subfield of the $m$-th cyclotomic field $\QQ(\zeta_m)$.

\subsection{The polynomial $D_N(w)$ and its generating function}\label{sec:determinant}

Recall \eqref{eq. DN} and set $G_N(w) := \frac{D_N'(w)}{D_N(w)}$.
By Lemma \ref{lem:vandermonde} and the Galois invariance argument of Lemma~\ref{lem:galois} below,
we conclude that $D_N(w)\in\ZZ[w]$, is monic, and has degree $M_N(m):=|\mathcal A(m,N)|=\binom{m-1}{N-1}$.

\begin{proposition}\label{prop:gf} Set $V_0(N,m):=2M_N(m)$, and recall that
 $c_{N,m}=Nm^{N-1}$.  Then
\begin{equation}
\sum_{n\ge0} V_n(N,m)\,\tau^n \;=\; 2\Big(M_N(m) - v\,\frac{D_N'(v)}{D_N(v)}\Big)\Big|_{v=c_{N,m}\tau}. \label{eq:gf}
\end{equation}
The identity holds as an identity of formal power series in $\mathbb C[[\tau]]$.  Analytically, both sides agree and the series converges absolutely, uniformly on compact subsets of $|\tau|<\min_x r_x/c_{N,m}$.
\end{proposition}

\begin{proof}
Since $D_N$ is monic of degree $M_N(m)$, we have that
\[
v\frac{D_N'(v)}{D_N(v)} = \sum_x \frac{v}{v-r_x} = M_N(m) + \sum_x \frac{r_x}{v-r_x}.
\]
Hence
$$
M_N(m)-vD_N'(v)/D_N(v) = -\sum_x \frac{r_x}{v-r_x} = \sum_x\sum_{k\ge0}v^k r_x^{-k-1}\cdot r_x,
$$
which is valid for $|v|$ smaller than every $r_x$.  By collecting powers of $v$,  we get that
$$
M_N(m)-vD_N'(v)/D_N(v) = \sum_{k\ge0}v^k\sum_x r_x^{-k}.
$$
Now, substitute that $v=c_{N,m}\tau$.  Since $2\sum_x r_x^{-n}=c_{N,m}^{-n}V_n(N,m)$, equation \eqref{eq:defn} implies \eqref{eq:gf},
which completes the proof.
\end{proof}

\subsection{Burnside's lemma and bracelet count}
Burnside's lemma is the classical orbit-counting theorem. If a finite group $G$ acts on a finite set $X$, the number of orbits equals the
average number of fixed points of the elements of $G$.  Precisely, if we let $X/G$ be the quotient space of $X$ obtained by the
action of $G$, then
$$
|X/G| \;=\; \frac{1}{|G|}\sum_{g\in G} |\mathrm{Fix}(g)|
\,\,\,\,\,
\textrm{\rm where}
\,\,\,\,\,
\mathrm{Fix}(g):=\{x\in X : g\cdot x = x\}.
$$

We will use Burnside's lemma as follows.  We take $X$ to be the set of $N$-element subsets of $\mathbb Z/m\mathbb Z$, assuming
$m > N$,
and $G=D_m$ to be the dihedral group of order $2m$. We will show that the quantity $r_x$ from \eqref{eq:r_definition} depends only on the $D_m$-orbit of the underlying set
$\{0,x_1,\dots,x_{N-1}\}$, so
we will be interested in bounding the number of $D_{m}$ orbits.  As it turns out, the
orbit count is exactly the classical count of binary bracelets, meaning the number of ways to
place $N$ marked beads among $m$ positions on a necklace up to action by the dihedral group $D_{m}$.
This analysis is given in detail in Lemma \ref{lem:dihedral} below.

The following proposition is the classical fixed weight, binary bracelet count. Equivalently, it counts incongruent cyclic
$N$-gons selected from $m$ equally spaced points, a problem solved by
Gupta~\cite{Gupta1979}.  Shevelev~\cite{Shevelev2004} gave a short proof and
made explicit the bijections with binary configurations and compositions;
the fixed-weight bracelet formulation is treated directly in
\cite{Shevelev2011}.  Via the cyclic-gap map, the same objects are dihedral
compositions of $m$ into $N$ positive parts, connecting the formula with the
composition literature \cite{KnopfmacherRobbins2010,KnopfmacherRobbins2013,HadjicostasZhang2017}; see also \cite{Ha23} for a generalization of the result below.

\begin{proposition} \label{prop:bracelet}
Let $B(N,m)$ be the number of orbits of $N$-element subsets of $\mathbb Z/m\mathbb Z$ under the
dihedral group $D_m$ --- the classical count of binary bracelets with $N$ marked beads among $m$
positions, meaning the number of ways to place $N$ marked beads among $m$ positions on a necklace
up to rotation and flipping. Then $B(N,m)$ is given by
\begin{equation}\label{eq:Burnside_count}
B(N,m) = \frac{1}{2m}\left(\sum_{k=0}^{m-1}\binom{d_k}{Nd_k/m}\mathbf 1_{m/d_k\mid N} + R(N,m)\right),
\qquad d_k := \gcd(m,k),
\end{equation}
with reflection term
$$
R(N,m) = \begin{cases}
m\binom{(m-1)/2}{\lfloor N/2\rfloor}, & m \text{ odd}, \\[4pt]
\dfrac{m}{2}\displaystyle\sum_{f=0}^{2}\binom{2}{f}\binom{(m-2)/2}{(N-f)/2} + \dfrac{m}{2}\binom{m/2}{N/2}\mathbf 1_{2\mid N}, & m \text{ even},
\end{cases}
$$
omitting any term whose lower binomial argument is not a nonnegative integer.
\end{proposition}

\begin{proof}
  The formula for $B(N,m)$ derived in \cite{Gupta1979} reads as
  \begin{equation}\label{eq:second}
 B(N,m)=\frac{1}{2m}\left(
 \sum_{q\mid\gcd(m,N)}\varphi(q)
 \binom{m/q}{N/q}+R_2(N,m)\right),
\end{equation}
where
\begin{equation}\label{eq:R2}
R_2(N,m)=
\begin{cases}
\displaystyle
m\binom{(m-1)/2}{\lfloor N/2\rfloor},
&m\text{ odd},\\[8pt]
\displaystyle
m\binom{m/2}{N/2},
&m\text{ even and }N\text{ even},\\[8pt]
\displaystyle
m\binom{(m-2)/2}{(N-1)/2},
&m\text{ even and }N\text{ odd}.
\end{cases}
\end{equation}
Here $\varphi$ denotes Euler's totient function. By putting $e=m/d_k$ for $k\neq 0$, noticing that there are exactly $\varphi(e)$ values of $k$ with $d_k=m/e$, and using elementary properties of binomial coefficients it is easy to see that the lead term in \eqref{eq:second} and the error term \eqref{eq:R2} coincide with the lead term in \eqref{eq:Burnside_count} and the error term $R(N,m)$.
\end{proof}

In the appendix we provide a self-contained, novel proof of Proposition \ref{prop:bracelet}.

\section{The resultant construction}\label{sec:resultant}

In this section we derive the auxiliary polynomial $W_N$ from \eqref{eq:W_definition} and prove that it has the properties
needed for our purposes. We then
give a complete, step-by-step proof of the resultant identity for $D_N(w)$.

\subsection{Motivation}\label{subsec:motivation}

The goal is to manufacture $D_N(w)=\prod_{x\in\mathcal A(m,N)}(w-r_x)$ out of finite, exact algebra on the roots of $Q(u):=u^m-1$ alone. Write $\mu_m:=\{u\in\CC: u^m=1\}$ for the group of $m$-th roots of unity, which is the set of roots of $Q$.

\begin{lemma}\label{lem:sign}
Let $u_0=1,u_1,\dots,u_{N-1}$ be pairwise distinct points on the unit circle. Then
\begin{equation}\label{eq:sign_formula}
\Delta(1,u_1,\dots,u_{N-1})^2 \;=\; (-1)^{\binom N2}\,\big|\Delta(1,u_1,\dots,u_{N-1})\big|^2 \prod_{a=1}^{N-1}u_a^{N-1}.
\end{equation}
\end{lemma}

\begin{proof}
Since $|u_a|=1$ for every $a$, $\bar u_a=u_a^{-1}$, so
\[
\overline{\Delta(u)} = \prod_{0\le a<b\le N-1}(\bar u_a-\bar u_b) = \prod_{a<b}\Big(\frac1{u_a}-\frac1{u_b}\Big) = \prod_{a<b}\frac{u_b-u_a}{u_au_b} = (-1)^{\binom N2}\frac{\Delta(u)}{\prod_{a<b}u_au_b}.
\]
Each index $a\in\{0,\dots,N-1\}$ occurs in exactly $N-1$ of the $\binom N2$ pairs, so
$$
\prod_{a<b}u_au_b=\prod_{a=0}^{N-1}u_a^{N-1}=\prod_{a=1}^{N-1}u_a^{N-1}$$
because $u_0=1$. Hence
\[
|\Delta(u)|^2 = \Delta(u)\overline{\Delta(u)} = (-1)^{\binom N2}\frac{\Delta(u)^2}{\prod_{a=1}^{N-1}u_a^{N-1}},
\]
and solving for $\Delta(u)^2$ gives the claim.
\end{proof}

Let us consider Lemma~\ref{lem:sign} at a point $x\in\mathcal A(m,N)$, so $u_a=u_a(x)=e^{2\pi ix_a/m}$ and $r_x=|\Delta(u(x))|^2$.
With this, \eqref{eq:sign_formula} becomes
\[
\Delta\big(1,u_1(x),\dots,u_{N-1}(x)\big)^2 = (-1)^{\binom N2}\,r_x\prod_{a=1}^{N-1}u_a(x)^{N-1}.
\]
One can view this computation as saying that $w=r_x$ is the unique root, in the variable $w$,  of the linear equation
$$
w\prod_a u_a^{N-1}-(-1)^{\binom N2}\Delta(1,u_1,\dots,u_{N-1})^2=0
$$
when $u_a=u_a(x)$. We can re-write this linear equation as \eqref{eq:W_definition}.
Further properties of \eqref{eq:W_definition} are as follows.

\begin{lemma}\label{lem:master}
Let $(u_1,\dots,u_{N-1})\in\mu_m^{N-1}$.  Then we have one of two possibilities.
\begin{enumerate}[label=(\alph*)]
\item If $1,u_1,\dots,u_{N-1}$ are pairwise distinct, then $(u_1,\dots,u_{N-1})$ is a permutation of $u(x)$ for a unique $x\in\mathcal A(m,N)$, and
\[
W_N(u_1,\dots,u_{N-1};w)=0 \iff w=r_x.
\]
\item If two or more of the $1,u_1,\dots,u_{N-1}$ coincide, then
\[
W_N(u_1,\dots,u_{N-1};w)=0 \iff w=0.
\]
\end{enumerate}
\end{lemma}

\begin{proof}
As a polynomial in $w$, $W_N$ is linear (see \eqref{eq:W_definition}), so for each fixed $(u_1,\dots,u_{N-1})$ it has exactly one root, namely
$$
w=(-1)^{\binom N2}\Delta(1,u_1,\dots,u_{N-1})^2/\prod_a u_a^{N-1}.
$$
In case (a) this root equals $r_x$, by the identity preceding \eqref{eq:W_definition}.
In case (b), $\Delta(1,u_1,\dots,u_{N-1})=0$ because two of the points $1,u_1,\dots,u_{N-1}$ coincide, so a
factor $u_a-u_b$ of $\Delta$ vanishes, so the unique root equals
$w=0$.
\end{proof}

\begin{corollary}\label{cor:count}
Of the $m^{N-1}$ tuples in $\mu_m^{N-1}$, the number in case (a) equals
$(N-1)!\binom{m-1}{N-1}$, and the number in case (b) equals $m^{N-1} - (N-1)!\binom{m-1}{N-1}$.
\end{corollary}

\begin{proof}
Case (a) tuples are exactly the ordered $(N-1)$-tuples of pairwise distinct elements of $\{1,\dots,m-1\}$, of which there are $(m-1)(m-2)\cdots(m-N+1)=(N-1)!\binom{m-1}{N-1}$.  Case (b) is the complement of case (a) in $\mu_m^{N-1}$, which accounts for the
remaining $m^{N-1}-(N-1)!\binom{m-1}{N-1}$ tuples.
\end{proof}

\begin{remark}\label{rem:counts}
When we divide the number of points in case (a) by $(N-1)!$ we get that
$$
\binom{m-1}{N-1}=|\mathcal A(m,N)|,
$$
as expected.  As part of the proof of Theorem~\ref{thm:resultant}, we show that the
$S_{N-1}$-action on the set of tuples in case (a) is free with
orbits of size $(N-1)!$, with distinct points in $\mathcal A(m,N)$ having disjoint orbits,
so the number of points in case (a) is $(N-1)!\cdot|\mathcal A(m,N)|$, again as expected.
\end{remark}

\begin{remark}\label{rem:rel_size}
The relative size of the two cases in Lemma \ref{lem:master} is not fixed.
Let $\pi_{N,m}$ denote the proportion of points in case (a), so then we have that
$$
\pi_{N,m}:=\#\{\text{case (a)}\}/m^{N-1}=\prod_{k=1}^{N-1}(1-k/m),
$$
which is the quantity familiar from the classical ``birthday problem''.  For example, if $N=N(m)\sim c\sqrt{m}$ for a fixed $c > 0$, we have that
$$
\pi_{N,m} \rightarrow e^{-c^2/2}
\,\,\,\,\,
\textrm{\rm as $m \rightarrow \infty$}.
$$

\end{remark}

\subsection{Computing the resultant}

The following is part (i) of the Main Theorem.

\begin{theorem}\label{thm:resultant}
Let $Q(u):=u^m-1$. Consider the iterated resultant
\[
\widetilde D_N(w) := \Res_{u_1}\Big(Q(u_1),\, \Res_{u_2}\big(Q(u_2),\dots,\Res_{u_{N-1}}(Q(u_{N-1}),W_N(u_1,\dots,u_{N-1};w))\dots\big)\Big).
\]
Then
\begin{equation}\label{eq:resultant_decomposition}
\widetilde D_N(w)= \varepsilon_{N,m}\cdot w^{\,m^{N-1}-M_N(m)(N-1)!}\, D_N(w)^{(N-1)!},
\end{equation}
where $\varepsilon_{N,m}=1$ for every $N\ge3$ and every $m$, and $\varepsilon_{2,m}=(-1)^{m-1}$.
\end{theorem}

\begin{proof}
Each $u_a$ ranges over the roots of the monic polynomial $Q$. By Remark~\ref{rem:iterated}, when applied with $f_a=Q$ for $a=1,\dots,N-1$ and $g=W_N$,
we get that
\begin{equation}\label{eq:full-product}
\widetilde D_N(w) = \prod_{(u_1,\dots,u_{N-1})\in\mu_m^{N-1}} W_N(u_1,\dots,u_{N-1};w),
\end{equation}
which is a product of $m^{N-1}$ linear factors in $w$.  More specifically,
each factor $W_N(u_1,\dots,u_{N-1};w)$ in \eqref{eq:full-product} is linear in $w$ with lead coefficient
$\prod_{a=1}^{N-1}u_a^{N-1}$.  Therefore,
$\widetilde D_N$ has degree $m^{N-1}$ with leading coefficient equal to
\[
\prod_{(u_1,\dots,u_{N-1})\in\mu_m^{N-1}}\ \prod_{a=1}^{N-1}u_a^{N-1} \;=\; \Big(\prod_{u\in\mu_m}u\Big)^{(N-1)^2m^{N-2}}.
\]
Indeed, for each fixed $a$ and each $u\in\mu_m$, there are $m^{N-2}$ tuples with $u_a=u$, hence the index $a$ contributes $\left(\prod_{u\in\mu_m} u\right)^{(N-1)m^{N-2}}$. Multiplying over the $N-1$ indices gives exponent $(N-1)^2m^{N-2}$.

Now, the constant term in $Q(X)=X^m-1=\prod_{u\in\mu_m}(X-u)$ gives $(-1)^m\prod_{u\in\mu_m}u=-1$, so then
\[
\prod_{u\in\mu_m}u = (-1)^{m-1},
\]
as expected.
Thus, the leading coefficient of $\widetilde D_N$ is $\big((-1)^{m-1}\big)^{(N-1)^2m^{N-2}}$. If $m$ is odd this is $1$ because $m-1$ is even.
If $m$ is even, the exponent $(N-1)^2m^{N-2}$ is even whenever $N\ge3$.  In either case, the leading coefficient is $1$.  It remains to consider
$N=2$, where in this case $m^{N-2}=m^0=1$.  Thus, the leading coefficient is $(-1)^{m-1}$.  With
all this, we have computed  $\varepsilon_{N,m}$.

Let us now compute the order of vanishing at $w=0$.  This is precisely the number of tuples in
case (b) of Lemma~\ref{lem:master}, and that is, indeed, the stated exponent in \eqref{eq:resultant_decomposition}.

It remains to consider the contribution to \eqref{eq:resultant_decomposition} of the tuples from case (a) of Lemma~\ref{lem:master}.
Any such tuple $(u_1,\dots,u_{N-1})$ is non-degenerate  (meaning it satisfies assumptions of case (a) of Lemma~\ref{lem:master}) exactly when $0,u_1,\dots,u_{N-1}$ are pairwise distinct mod $m$.
Equivalently, this occurs when some permutation of $(u_1,\dots,u_{N-1})$ lies in the alcove $\mathcal A(m,N)$.
Thus, the symmetric group $S_{N-1}$ acts on non-degenerate tuples by permuting the labels $u_1,\dots,u_{N-1}$.  Clearly, this action is free  and transitive on the fiber over each alcove point, so every orbit has size exactly $(N-1)!$. Moreover $\Delta(1,u_1,\dots,u_{N-1})^2$, hence $W_N$ itself, is invariant
under this action, being a symmetric function of $u_1,\dots,u_{N-1}$.  Furthermore, all $(N-1)!$ tuples in the orbit of a
given $x$ contribute the identical factor $(w-r_x)$. By collecting the non-degenerate tuples by orbit, we get that
\begin{align*}
 \widetilde D_N(w) &= \varepsilon_{N,m}\cdot w^{\,m^{N-1}-M_N(m)(N-1)!} \prod_{x\in\mathcal A(m,N)}(w-r_x)^{(N-1)!} \\&=\varepsilon_{N,m}\cdot w^{\,m^{N-1}-M_N(m)(N-1)!} D_N(w)^{(N-1)!},
\end{align*}
where we used that the number of degenerate tuples (corresponding to part (b) of Lemma~\ref{lem:master}) is $m^{N-1}-M_N(m)(N-1)!$.

\end{proof}

\begin{corollary}\label{cor:extract}
For $N\ge3$ and $w\ne0$, we have that
\[
D_N(w)^{(N-1)!} = \frac{1}{w^{\,m^{N-1}-M_N(m)(N-1)!}}\ \Res_{u_1}\Big(Q(u_1),\dots,\Res_{u_{N-1}}(Q(u_{N-1}),W_N(\cdot;w))\dots\Big).
\]
\end{corollary}

\subsection{Comparing with \cite{JKS26}}
The resultant construction in \cite{JKS26} for $\textrm{\rm SU}(3)$ was different because we introduced
a different auxiliary polynomial $W$.  Let us now compare the two constructions.  The end result is the
following.  Let $D^{(1)}(w)$ be the resultant constructed in \cite{JKS26} and let $D_3(w)$ be the resultant computed
for $\textrm{\rm SU}(3)$,
then
\begin{equation}\label{eq:resultant_comparisons}
D^{(1)}(w) \;=\; 4^{-\binom{m-1}{2}}\, D_3(4w^2)
\,\,\,\,\,
\textrm{\rm or}
\,\,\,\,\,
D_3(w) \;=\; 4^{\binom{m-1}{2}}\, D^{(1)}\!\left(\frac{\sqrt{w}}{2}\right).
\end{equation}

One of the first differences to note between \cite{JKS26} and the above computations is that
we employed different polynomials $Q$.  In this paper, the elimination theory applied to the
roots of unity, meaning that we set $Q(u) = u^m - 1$ while fixing the basepoint $x_0 := 0$ in the alcove.
In \cite{JKS26}, we came from the point of view of a circulant graph, which led to considering
the so-called twisted roots of
$Q(u) = u^m - i^m$ whose roots $z_j = i\zeta_m^j$ implement the character twist
$\beta = (m/4,m/4,m/4)$ directly into the choice of $Q$ and the constraint that three
roots of unity satisfy $z_{j_1}z_{j_2}z_{j_3} = -i$.

More specifically,
let $W_{N}$ denote the polynomial in this paper, and $W$ denote the polynomial from \cite{JKS26}.  Recall
that
$$
W(u,v,w) = iu^2v^2 + u^2v + uv^2 - 2wuv + u+v-i.
$$
The polynomial $W$
is linear in $w$ with lead coefficient $-2uv$, and its root is instead the signed quantity
$r_j := \mathrm{Re}(u)+\mathrm{Re}(v)+\mathrm{Re}(z_3) = 3\Lambda_{m,\beta}(j)$, a sum of cosines.
So, by comparing when $N=3$, the roots $\{r_{x}\}$ of $W_{N}$ correspond to the roots $\{r_{j}\}$ of $W$
through the relation $r_x = 4r_j^2$.  The lead coefficient is obtained by computing the product of coefficients
from $W$ and the computation in given in Theorem \ref{thm:resultant}, which can also be written as
$$
4^{\binom{m-1}{2}} = 2^{(m-1)(m-2)}.
$$

\section{The unreduced recursion}\label{sec:recursion}

In the previous section we proved part (i) of the Main Theorem.  In this section we will prove parts
(ii), (iii) and (iv).

\begin{theorem}\label{thm:recursion}
Fix $N\geq 2$ and $m> N$. Recall that
$M_N(m) = \binom{m-1}{N-1}$ and $c_{N,m} = Nm^{N-1}$.
Let
$$
Q_N(\tau):=D_N(c_{N,m}\tau)
$$
and
$$
P_N(\tau):=2\big(M_N(m)D_N(c_{N,m}\tau)-c_{N,m}\tau D_N'(c_{N,m}\tau)\big),
$$
which are polynomials in $\tau$ of degrees $M_N(m)$ and at most $M_N(m)-1,$ respectively. Then
\[
\sum_{n\ge0}V_n(N,m)\tau^n = \frac{P_N(\tau)}{Q_N(\tau)}.
\]
Furthermore, by writing
$$
Q_N(\tau)=\sum_k q_k\tau^k
\,\,\,\,\,
\textrm{\rm and}
\,\,\,\,\,
P_N(\tau)=\sum_k p_k\tau^k,
$$
then the sequence $V_n(N,m)$ of Verlinde
sums for fixed $N$ and $m$ satisfy the linear recursion
\[
\sum_{k=0}^{\min(n,M_N(m))} q_k\,V_{n-k}(N,m) = \begin{cases} p_n, & 0\le n\le M_N(m),\\ 0, & n>M_N(m).\end{cases}
\]
of order at most $M_N(m)$.
\end{theorem}

\begin{proof}
The proof is immediate from Proposition \ref{prop:gf} upon clearing the denominator $D_N(c_{N,m}\tau)$ and matching coefficients of $\tau^n$.
\end{proof}

\begin{remark}
The recursion above has order $\leq M_N(m)=\binom{m-1}{N-1}$, which grows quickly in $m$.  In section \ref{sec:reduction}
we will study how to shorten the recursion length.
\end{remark}

\begin{example}[$N=3$, $m=4$]\label{ex:N3m4}
Here $\binom{3}{2}=3$, so $(-1)^{\binom{3}{2}}=-1$ and
$$
W_3(u_1,u_2;w) = w\,u_1^2u_2^2 + \Delta(1,u_1,u_2)^2
= w\,u_1^2u_2^2 + \bigl[(1-u_1)(1-u_2)(u_1-u_2)\bigr]^2.
$$
The alcove $A(4,3)=\{(1,2),(1,3),(2,3)\}$ has $3=\binom{3}{2}$ points. At $x=(1,2)$, so $u_1=i$, $u_2=-1$, and
$$
\Delta(1,i,-1) = (1-i)(1-(-1))(i-(-1)) = (1-i)\cdot 2\cdot(1+i) = 2(1-i^2)=4,
$$
so $r_{(1,2)}=|\Delta|^2=16$. Indeed, $W_3(i,-1;16)=16\cdot i^2(-1)^2+4^2=-16+16=0$, confirming $w=16$ is the
root predicted by Lemma \ref{lem:sign}.
The same computation at $x=(1,3),(2,3)$ also gives $r_x=16$, so all three alcove points coincide,
forcing $D_3(w)=(w-16)^3$.

Let us check this computation using Theorem \ref{thm:resultant}.  Since $(N-1)!=2$, $M_3(4)=\binom{3}{2}=3$, $m^{N-1}=16$, so the exponent of $w$
in Theorem \ref{thm:resultant} is
$16-3\cdot 2=10$. Thus, the iterated resultant gives
$$
\widetilde D_3(w) = w^{10}(w-16)^6,
$$
exactly matching $\varepsilon_{3,4}\,w^{10}D_3(w)^2$ with $\varepsilon_{3,4}=1$ and $D_3(w)=(w-16)^3$.
\end{example}

\begin{example}[$N=4$, $m=5$]\label{ex:N4m5}
Here $\binom{4}{2}=6$, so $(-1)^{\binom{4}{2}}=1$ and
$$
W_4(u_1,u_2,u_3;w) = w\,u_1^3u_2^3u_3^3 - \Delta(1,u_1,u_2,u_3)^2.
$$
The alcove $A(5,4)=\{(1,2,3),(1,2,4),(1,3,4),(2,3,4)\}$ has $4=\binom{4}{3}$ points. At $x=(1,2,3)$, with $\zeta:=e^{2\pi i/5}$ and $u_a=\zeta^a$,
one computes
$$
\Delta(1,u_1,u_2,u_3)^2 = 125\,\zeta^3,
$$
so $r_{(1,2,3)}=|\Delta|^2=125$. Indeed, $W_4(u_1,u_2,u_3;125)=125\cdot\zeta^{3(1+2+3)}-125\zeta^3=125\zeta^{18}-125\zeta^3=0$
since $\zeta^{18}=\zeta^3$, confirming $w=125$ is the root predicted by Lemma \ref{lem:sign}. The same computation at the
other three alcove points also gives $r_x=125$, so all four alcove points coincide, thus
$$
D_4(w) = (w-125)^4 = w^4-500w^3+93750w^2-7812500w+244140625.
$$

Let us check this computation using Theorem \ref{thm:resultant}.  Since $(N-1)!=6$, $M_4(5)=\binom{4}{3}=4$, $m^{N-1}=125$,
so the predicted power of $w$ is $125-4\cdot 6=101$. Direct symbolic computation of the iterated resultant gives
$$
\widetilde D_4(w) = w^{101}(w-125)^{24},
$$
exactly matching $\varepsilon_{4,5}\,w^{101}D_4(w)^6$ with $\varepsilon_{4,5}=1$.
\end{example}

\begin{example}[$N=6$, $m=8$]\label{ex:N6m8}
Here $\binom{6}{2}=15$, so $(-1)^{\binom{6}{2}}=-1$ and
$$
W_6(u_1,\dots,u_5;w) = w\prod_{a=1}^5 u_a^5 + \Delta(1,u_1,\dots,u_5)^2.
$$
The alcove $A(8,6)$ has $M_6(8)=\binom{7}{5}=21$ points. At $x=(1,2,3,4,6)$, with $\zeta:=e^{2\pi i/8}$ and $u_a=\zeta^{x_a}$, one computes
$$
\Delta(1,u_1,\dots,u_5)^2 = -8192,
$$
so $r_x=|\Delta|^2=8192$.  Specifically, $W_6(u_1,\dots,u_5;8192)=8192\cdot\zeta^{5(1+2+3+4+6)}+(-8192)=8192\cdot\zeta^{80}-8192=0$
since $80\equiv 0\pmod 8$, confirming $w=8192$ is the root predicted by Lemma \ref{lem:sign}. Unlike the two previous examples,
the $21$ alcove points do not all coincide.  A direct evaluation of $\{r_x\}_{x\in A(8,6)}$ shows that
\begin{equation}\label{eq:N6m8_D_formula}
D_6(w) = (w-8192)^6(w-16384)^3\bigl(w^2-16384w+33554432\bigr)^6,
\end{equation}
which is a monic degree-$21$ polynomial (note that $6+3+2\cdot 6=21=M_6(8)$), whose irreducible quadratic factor has roots
$8192\mp 4096\sqrt2$.

For the computational check, $(N-1)!=120$, $m^{N-1}=8^5=32768$, so the power of $w$ is $32768-21\cdot 120=30248$. Direct symbolic computation of the iterated resultant gives
$$
\widetilde D_6(w) = w^{30248}(w-8192)^{720}(w-16384)^{360}\bigl(w^2-16384w+33554432\bigr)^{720},
$$
matching $\varepsilon_{6,8}\,w^{30248}D_6(w)^{120}$ with $\varepsilon_{6,8}=1$.
\end{example}

\begin{remark}
Using the computations of Examples \ref{ex:N3m4}, \ref{ex:N4m5} and \ref{ex:N6m8}, one can employ
Theorem \ref{thm:recursion} to compute the Verlinde sums in all of these cases.  Example \ref{ex:N3m4}
confirms, by different means, the evaluations of Verlinde sums when $N=3$ and $m=4$ which was
given in \cite{JKS26}.
\end{remark}

\begin{theorem}\label{thm:binet}
Let $\{a\}$ denote the distinct values among $\{r_x\}$ for $x\in\mathcal A(m,N)$ counted with multiplicity $\mu(a)$. For every $n\ge1$,
\[
V_n(N,m) = 2\big(Nm^{N-1}\big)^{n}\sum_a \mu(a)\,a^{-n} = 2\sum_a \mu(a)\,y_a^{\,n}
\,\,\,\,\,
\textrm{\rm where}
\,\,\,\,\,
y_a:=\frac{Nm^{N-1}}{a}.
\]
\end{theorem}

\begin{proof}
The proof is immediate from \eqref{eq:defn} when grouping the sum $\sum_x r_x^{-n}$ by the distinct value of $r_x$.
\end{proof}

\begin{remark}
The computations in Example \ref{ex:N6m8} combine with Theorem \ref{thm:binet} to give a precise formula
for Verlinde sums in the case $N=6$ and $m=8$.  Specifically, when factoring \eqref{eq:N6m8_D_formula}
we get the succinct formula that
$$
V_n(6,8) = 2\Bigl[\,6\cdot 24^n + 3\cdot 12^n + 6(48+24\sqrt2)^n + 6(48-24\sqrt2)^n\,\Bigr]
\,\,\,\,\,
\textrm{\rm for any $n\ge 1$.}
$$
This example will be fully developed below, as well as other instances of the Main Theorem.
\end{remark}

\section{Reduction of the length of the recursion}\label{sec:reduction}

Note that in the examples in Section \ref{sec:recursion} all factors in polynomials $D_{N}$ had relatively
large exponents, so in fact the length of the recursive formulas for these Verlinde sums is shorter than
$M_N(m)$. The purpose of this section is to obtain a better bound for the length of the recursive sequence.

\subsection{The field bound}

\begin{lemma}\label{lem:galois}
For every $t\in(\ZZ/m\ZZ)^\times$, the map $\sigma_t:\zeta_m\mapsto\zeta_m^t$ permutes the $N$-element subsets of $\ZZ/m\ZZ$ containing zero and hence induces a permutation of $\mathcal A(m,N)$.  Write
$\sigma_t\cdot x$ for the unique alcove point whose underlying set is $\{tx_a \bmod m\}$.  Then one has that $r_{\sigma_t\cdot x}=\sigma_t(r_x)$.
Consequently $D_N(w)\in\ZZ[w]$, and every root of $D_N$ lies in $\QQ(\zeta_m)^+$, so every irreducible factor of $D_N$ over $\QQ$ has degree
dividing $\varphi(m)/2$.
\end{lemma}

\begin{proof}
Multiplication by $t$ for $t\in(\ZZ/m\ZZ)^\times$
is a bijection of $\ZZ/m\ZZ$ fixing $0$, so the map carries the underlying $N$-element set $\{0,x_1,\dots,x_{N-1}\}$ to
another $N$-element subset of $\ZZ/m\ZZ$ containing $0$, which after sorting is the underlying set of a unique point
$\sigma_t\cdot x\in\mathcal A(m,N)$.  Since $r_x=\prod_{a<b}|e^{2\pi ix_a/m}-e^{2\pi ix_b/m}|^2$ is a polynomial with
integer coefficients in the $\zeta_m$-powers $e^{2\pi ix_a/m}$ and their conjugates.  By applying $\sigma_t$ termwise
gives that $\sigma_t(r_x)=r_{\sigma_t\cdot x}$. Thus $\{r_x\}_{x\in\mathcal A(m,N)}$ is stable, as a multiset with multiplicity,
under the full Galois group $\mathrm{Gal}(\QQ(\zeta_m)/\QQ)$, so $D_N(w)=\prod_x(w-r_x)$ has coefficients fixed by this
group.  Hence $D_N(w)\in\QQ[w]$.  Since $D_{N}(w)$ is monic with algebraic integer roots, by Lemma~\ref{lem:vandermonde},
$D_N(w)\in\ZZ[w]$. Since $r_x$ is real, $r_x\in\QQ(\zeta_m)^+$, and $\mathbb Q(\zeta_m)^+/\mathbb Q$ is a finite Galois extension, the degree of the minimal polynomial of $r_x$ (which is an irreducible factor of $D_N$) equals the size of its Galois orbit, hence divides $[\mathbb Q(\zeta_m)^+:\mathbb Q]=\varphi(m)/2$.
\end{proof}

\subsection{Dihedral symmetry}\label{subsec:dihedral}

The Galois action of Lemma~\ref{lem:galois} only produces conjugate values. It does not, by itself, identify when any two values
in $\{r_x\}$ are equal. We now study when $r_x$ is invariant under rotating and reflecting the $m$-point discrete circle acting on the configuration
$\{0,x_1,\dots,x_{N-1}\}$.  In other words, we are studying the fixed points of the action by the dihedral group $D_{m}$.

\begin{lemma}\label{lem:dihedral}
Extend $r_x$ to a function $\rho$ of arbitrary $N$-element subsets $T\subset\ZZ/m\ZZ$ by
$$
\rho(T):=|\Delta(e^{2\pi it/m}:t\in T)|^2,
$$
so $r_x=\rho(\{0,x_1,\dots,x_{N-1}\})$. Then for every $c\in\ZZ/m\ZZ$ and $\varepsilon\in\{\pm1\}$,
\[
\rho(\varepsilon T+c) = \rho(T)
\,\,\,\,\,
\textrm{\rm where}
\,\,\,\,\,
\varepsilon T+c := \{(\varepsilon t+c)\bmod m : t\in T\}.
\]
\end{lemma}

\begin{proof}
Write $u_t:=e^{2\pi it/m}$. Translation by $\varepsilon=1$ gives $u_{t+c}=e^{2\pi ic/m}u_t$, so $u_{t+c}-u_{t'+c}=e^{2\pi ic/m}(u_t-u_{t'})$
picks up a common factor which has modulus one, so $\rho(T+c)=\rho(T)$. Reflection  by $\varepsilon=-1$ with $c=0$ gives
$u_{-t}=\overline{u_t}$, so $u_{-t}-u_{-t'}=\overline{u_t-u_{t'}}$, and $|\overline z|=|z|$ implies $\rho(-T)=\rho(T)$. The general case
follows by considering compositions of these actions.
\end{proof}

\begin{theorem}\label{thm:minrec}
Let $K_{\min}(N,m)$ denote the number of distinct values among $\{r_x\}_{x\in\mathcal A(m,N)}$. The minimal-order linear recursion with constant coefficients annihilating $(V_n(N,m))_{n\ge1}$ has order exactly $K_{\min}(N,m)$.
\end{theorem}

\begin{proof}
Write $D_N(w)=\prod_a(w-a)^{\mu(a)}$ over distinct roots $a$, with multiplicities $\mu(a)$. Then $G_N=D_N'/D_N=\sum_a \mu(a)/(w-a)$
has only simple poles, with residue $\mu(a)$. By Proposition \ref{prop:gf}, $M_N(m)-vD_N'(v)/D_N(v) = \sum_a\mu(a)\,v/(a-v)$,
so $\sum_n V_n(N,m)\tau^n$ is, after the substitution $v=c_{N,m}\tau$, a sum of $K_{\min}(N,m)$ simple geometric series.
As such, when combining into a single rational function, the denominator of the reduced rational function  $\sum_n V_n(N,m)\tau^n$ has degree $K_{\min}(N,m)$ and no repeated roots.
\end{proof}

\begin{corollary}\label{cor:kmin-bound}
For fixed $N$ and $m$ we have $K_{\min}(N,m)\le B(N,m)$.
\end{corollary}

\begin{proof}
By Lemma~\ref{lem:dihedral}, $r_x$ depends only on the $D_m$-orbit of the underlying set
$\{0,x_1,\dots,x_{N-1}\}$, and every $D_m$-orbit of $N$-element subsets of $\mathbb Z/m\mathbb Z$ is
met by some point of $\mathcal A(m,N)$.  Hence,
the number of distinct values among $\{r_x\}_{x\in \mathcal A(m,N)}$ is at most $B(N,m)$, the number of such
orbits computed in Proposition~\ref{prop:bracelet}. Since $K_{\min}(N,m)$ equals this count of
distinct values by Theorem~\ref{thm:minrec} above, the claim follows.
\end{proof}

\begin{remark}
Among our examples below are instances where $B(N,m) \neq K_{\min}(N,m)$; see, for example,
Table~\ref{tab:bracelet} where a representative sample is recorded.  In creating this table,
we also verified Proposition~\ref{prop:bracelet} in two ways.  First, by direct orbit enumeration, matching the closed form exactly
for every $(N,m)$ with $2\le N\le 6$, $N<m\le18$.  Second,  by comparing $B(N,m)$ against $K_{\min}(N,m)$ which was computed from $60$-digit
numerical evaluation of $\{r_x\}$, for $67$ pairs $(N,m)$ in the same range.  The equality $K_{\min}(N,m)=B(N,m)$ holds for $46$ of these $67$,
in particular for every $N=2$ case, and all cases with $m$ prime except pairs $(4,13)$, $(6,13)$ and (6,17).
\end{remark}

\begin{table}[h]
\centering
\begin{tabular}{cccccc}
\toprule
$N$ & $m$ & $M_N(m)$ & $B(N,m)$ & $K_{\min}(N,m)$ & gap \\
\midrule
4 & 7  & 20  & 4   & 4   & 0 \\
4 & 9  & 56  & 10  & 10  & 0 \\
4 & 11 & 120 & 20  & 20  & 0 \\
4 & 12 & 165 & 29  & 25  & 4 \\
4 & 13 & 220 & 35  & 34  & 1 \\
4 & 17 & 560 & 84  & 84  & 0 \\
5 & 9  & 70  & 10  & 10  & 0 \\
5 & 13 & 495 & 57  & 57  & 0 \\
6 & 11 & 252 & 26  & 26  & 0 \\
6 & 13 & 792 & 76  & 74  & 2 \\
\bottomrule
\end{tabular}
\caption{The bracelet bound $B(N,m)$ versus the true minimal recursion order $K_{\min}(N,m)$.}
\label{tab:bracelet}
\end{table}

\begin{remark}\label{rem:growth}
The examples of \S\ref{sec:examples} feature striking collapses.  For example, with $N=5$ and $m=6$, the full iterated resultant has degree $6^4=1296$, whereas $D_5$ has degree $M_5(6)=5$ and only one distinct root, meaning that $K_{\min}(5,6)=1$.   It is natural to ask whether the dihedral reduction of
Proposition~\ref{prop:bracelet} improves the order of growth of the recursion length as a
function of $m$, for $N$ fixed, or merely is constant. The upper bound \eqref{K bound} indicates that savings is at most a constant. For every fixed
$N\ge3$, $K_{\min}(N,m)$ grows at most at the same polynomial rate in $m$ as $M_N(m)$ itself, namely we have that
\begin{equation}\label{K bound}
K_{\min}(N,m)\le \frac{M_N(m)}{2N}+O\bigl(m^{\lfloor N/2\rfloor}\bigr).
\end{equation}
Our computations suggest that this bound is often sharp or nearly sharp, but a matching lower bound for $K_{\min}$ is not proved here.

Let us prove the bound \eqref{K bound}. We start by comparing the lead terms.  Since $\binom{m-1}{N-1}=\frac Nm\binom mN$,
$$
M_N(m) = \frac{N}{m}\binom{m}{N} \qquad\text{equivalently}\qquad \frac{M_N(m)}{2N} =\frac1{2m}\binom mN.
$$
In the sum of Proposition~\ref{prop:bracelet}, the term $k = 0$ has $d_0=\gcd(m,0)=m$, so it
contributes $\binom{d_0}{Nd_0/m}=\binom mN$. By the identity above, this single term, divided by $2m$, equals
$M_N(m)/(2N)$ precisely, not asymptotically.

Let us now show that all other terms are of lower order.
For $k\ne0$, $d_k=\gcd(m,k)$ is a  divisor
of $m$, so $e:=m/d_k$ is an integer $\ge2$, and the term vanishes unless $e\mid N$. Since $N$ is
fixed, $e$ ranges over the divisors of $N$ exceeding $1$.  There are at most $d(N)-1$ divisors, where $d$ is the divisor counting function.
For each such $e$ dividing $m$, exactly $\varphi(e)$ values of $k$ give $d_k=m/e$, with each contributing
$\binom{m/e}{N/e}=O(m^{N/e})$. By writing $p\ge2$ for the smallest prime factor of $N$, the total
non-identity contribution to the first term is therefore $O(m^{N/p})=O(m^{N/2})$, a bounded number of terms
each of size $O(m^{N/2})$. The reflection term
$R(N,m)$ is bounded the same way directly from its closed form in Proposition~\ref{prop:bracelet}.
Indeed, in each parity case the bound is a fixed multiple of $m$ times a binomial coefficient whose lower index is
$O(1)$ and whose upper index is $\Theta(m)$, giving $R(N,m)=O\bigl(m^{\lfloor N/2\rfloor+1}\bigr)$.

When combining these bounds and dividing by $2m$, we get that
$$
B(N,m) \;=\; \frac{M_N(m)}{2N} \;+\; O\bigl(m^{N/2-1}\bigr) \;+\; O\bigl(m^{\lfloor N/2\rfloor}\bigr)
\;=\; \frac{M_N(m)}{2N} + O\bigl(m^{\lfloor N/2\rfloor}\bigr).
$$
Since $K_{\min}(N,m)\le B(N,m)$ this proves \eqref{K bound}. Moreover, this shows that the dihedral reduction of the recurrence order described above asymptotically saves a factor of about $2N$.


This is visible already in Table~\ref{tab:bracelet}: for fixed $N$, the ratio $M_N(m)/B(N,m)$ climbs
steadily toward $2N$ as $m$ increases and does not stabilize early. For $N=4$ it runs
$5.00,\,5.60,\,6.00,\,6.29,\,6.67$ at $m=7,9,11,13,17$, visibly heading toward $2N=8$; for $N=6$ it
runs $9.69,\,10.42$ at $m=11,13$, heading toward $2N=12$.
\end{remark}

\begin{example}[$N=4$, $m=13$]\label{ex:resonance-N4m13}
Every gap in Table \ref{tab:bracelet} reflects at least one pair of $D_m$-inequivalent alcove
configurations that nonetheless share the exact same value of $r_x$, which is
a coincidence an upper bound \eqref{K bound} does not predict. At $N=4$, $m=13$, exactly one such coincidence occurs,
between the dihedral orbits of $\{0,1,4,6\}$ and $\{0,1,3,9\}$, which are neither a rotation nor a reflection
of the other, and the shared value is strikingly clean, namely
$$
r_{\{0,1,4,6\}} = r_{\{0,1,3,9\}} = 13 = m.
$$

This is not an isolated numerical accident.  It reflects arithmetic structure in the prime
factorization of $r_x$ over $\mathcal O_{K^+}$, where $K^+ := \mathbb Q(\zeta_m)^+$, tied to the ramification
of $m=13$ in $ \mathbb Q(\zeta_{13})$. We do not pursue the number-theoretic study of the
factorization of $\{r_x\}$ here, including when such resonances imply exact coincidences versus merely
constraining valuations. We will leave the broader number-theoretic study of $\{r_x\}$ for a
future investigation.
\end{example}

\section{Examples}\label{sec:examples}

In every example below, the distinct values of $r_x$ were identified exactly in $\mathbb Q(\zeta_m)^+$ from $60$ significant digits numerics, after which $D_N$ was assembled from the defining product $\prod_x(w-r_x)$ in exact integer arithmetic.  The resulting
Binet formula and reduced recursion were verified against direct evaluation of \eqref{eq:defn} to at least $15$ significant digits
for several values of $n$ beyond the minimal order of the recursion equation. To save space, in many parts simply will
list the results.  For $N=2$, the results were compared with \cite{JKS24}, and for $N=3$ with \cite{JKS26}.

\begin{example}{$N=2$. }\label{subsec:N2}
Here $\mathcal A(m,2)=\{1,\dots,m-1\}$, $r_x=4\sin^2(\pi x/m)$, $c_{2,m}=2m$.

\smallskip
\emph{$m=5$:} $D_2(w)=(w^2-5w+5)^2$. The distinct roots of $D_{2}$ are $w=\tfrac{5\pm\sqrt5}2$, each of multiplicity $2$. $K_{\min}=2$.
The reduced recursion is
\[
V_n(2,5) = 10\,V_{n-1}(2,5) - 20\,V_{n-2}(2,5)
\,\,\,\,\,
\textrm{\rm for $n>2$,}
\]
with seeds $V_1(2,5)=40$, $V_2(2,5)=240$. This reproduces $V_3(2,5)=1600$, $V_4(2,5)=11200$.   Binet's formula is
$$
V_n(2,5) = 4\left[(5+\sqrt5)^n + (5-\sqrt5)^n\right]
\,\,\,\,\,
\textrm{\rm for $n\ge 1$.}
$$

\smallskip
\emph{$m=7$:} $D_2(w)=(w^3-7w^2+14w-7)^2$, a single irreducible cubic factor of degree $=\varphi(7)/2=3$ and
of multiplicity $2$. $K_{\min}=3$. The reduced recursion is
\[
V_n(2,7) = 28\,V_{n-1}(2,7) - 196\,V_{n-2}(2,7) + 392\,V_{n-3}(2,7)
\textrm{\rm for $n>3$,}
\]
with seeds $V_1(2,7)=112$, $V_2(2,7)=1568$, $V_3(2,7)=26656$. This reproduces $V_4(2,7)=482944$, $V_5(2,7)=8912512$.
Binet's formula is
\[
V_n(2,7) = 4\left(y_0^n + y_1^n + y_2^n\right)
\,\,\,\,\,
\textrm{\rm for $n\ge 1$,}
\]
where $y_0, y_1, y_2$ are the roots of $y^3 - 28y^2 + 196y - 392 = 0$.
\end{example}

\medskip
\begin{example}{$N=3$.}\label{subsec:N3}
Here $c_{3,m}=3m^2$.

\smallskip
\emph{$m=5$:} $D_3(w)=(w^2-25w+125)^3$ which has the distinct roots $w=\tfrac{25\pm5\sqrt5}2$, each of multiplicity $3$. $K_{\min}=2$. With $y=\tfrac{75}w=\tfrac{3(5\mp\sqrt5)}2$, we get that
\[
V_n(3,5) = 6\left[\Big(\tfrac{3(5+\sqrt5)}2\Big)^{\!n} + \Big(\tfrac{3(5-\sqrt5)}2\Big)^{\!n}\right]
\,\,\,\,\,
\textrm{\rm for $n\ge 1$.}
\]
Equivalently the order $2$ recursion is
$$
V_n(3,5)=15V_{n-1}(3,5)-45V_{n-2}(3,5)
\,\,\,\,\,
\textrm{\rm for $n\ge 3$,}
$$
with seeds $V_1(3,5)=90$, $V_2(3,5)=810$. This reproduces $V_3(3,5)=8100$, $V_4(3,5)=85050$.

\smallskip
\emph{$m=6$:} $D_3(w)=(w-3)^3(w-12)^6(w-27)$. $K_{\min}=3$. The reduced recursion is
\[
V_n(3,6) = 49\,V_{n-1}(3,6) - 504\,V_{n-2}(3,6) + 1296\,V_{n-3}(3,6), \qquad n>3,
\]
with seeds $V_1(3,6)=332$, $V_2(3,6)=8780$, $V_3(3,6)=288812$. This reproduces $V_4(3,6)=10156940$.
The Binet formula is
\[
V_n(3,6) = 6\cdot 36^n + 12\cdot 9^n + 2\cdot 4^n
\,\,\,\,\,
\textrm{\rm for $n\ge 1$.}
\]
\end{example}

\medskip
\begin{example}{$N=4$.}\label{subsec:N4}
Here $c_{4,m}=4m^3$.

\smallskip
\emph{$m=5$:}
When $N=4,m=5$, by Theorem~\ref{thm:resultant}, the iterated resultant has degree $m^{N-1}=5^3=125$. Corollary~\ref{cor:extract} removes $w^{125-M_4(5)\cdot3!}=w^{125-24}=w^{101}$ and takes a $3!=6$-th root of the rest, leaving a degree $4$ polynomial
\[
D_4(w) = (w-125)^4 = w^4 - 500\,w^3 + 93750\,w^2 - 7812500\,w + 244140625.
\]
Also, $M_{4}(5)=4$.  The
unreduced order $4$ recursion from Theorem~\ref{thm:recursion} we obtain the coefficients
$$
(p_0, p_1, p_2, p_3) = (1953125000,\ -23437500000,\ 93750000000,\ -125000000000)
$$
and
\[
(q_0,q_1,q_2,q_3,q_4) = (244140625,\ -3906250000,\ 23437500000,\ -62500000000,\ 62500000000).
\]
Note that these values already have  $9$ to $11$ digits. However, by writing $D_4(500\tau) = 244140625\,(4\tau-1)^4=  5^{12}(4\tau-1)^4$
makes it transparent that the $q_k$'s are the binomial coefficients of $(4\tau-1)^4$ scaled by $5^{12}$. The corresponding numerator is
\begin{align*}
P_4(\tau) &= -1953125000\,(4\tau-1)^3
\\&= 1953125000 - 23437500000\,\tau + 93750000000\,\tau^2 - 125000000000\,\tau^3.
\end{align*}
The full unreduced order 4 recursion reads
\[
244140625\,V_n - 3906250000\,V_{n-1} + 23437500000\,V_{n-2} - 62500000000\,V_{n-3} + 62500000000\,V_{n-4}=0
\]
with
$V_1(4,5) = 32$, $V_2(4,5) = 128$,  $V_3(4,5) = 512$, and $V_4(4,5) = 2048$.  The normalized recursion formula becomes
\[
V_n(4,5) - 16\,V_{n-1}(4,5) + 96\,V_{n-2}(4,5) - 256\,V_{n-3}(4,5) + 256\,V_{n-4}(4,5) = 0
\,\,\,\,\,
\textrm{\rm for $n >4$.}
\]

However, because $D_4$ has a single root of multiplicity 4, the factor $(4\tau-1)^3$ cancels between $P_4$ and $Q_4$, and the generating function
collapses immediately to
\[
\sum_{n\ge0} V_n(4,5)\,\tau^n = \frac{8}{1-4\tau},
\]
thus recovering $V_n(4,5) = 8\cdot 4^n$.  This example is the most stark illustration of Theorem 4.7 subsuming Theorem 4.1 when $K_{\min}=1$.

\smallskip
\text{\it Remark.}
The case $N=4$ and $m=5$ has level $k=1$, since $k = m-N$.  In the intriguing article \cite{Za96}, Zagier pointed out
that direct computations which seek trigonometric simplification of $V_{n}(4,m)$ becomes impractical.  Indeed,
any reader would appreciate the characteristic candor and wit from \cite{Za96} upon reading the following.

\smallskip
\begin{quote}
\itshape ``The method of calculation also works for the Verlinde formula of higher rank $n$ \dots and expresses the Verlinde number in
this case as the coefficient of $\prod(\sin x_i)^{2g-2}$ in a certain trigonometric function of $n(n-1)/2$ variables $x_i$. However,
already for $n=4$ the form obtained for this trigonometric function (of $6$ variables) was as a sum of about $80$ terms, and the
algebra of putting these terms over a common denominator and simplifying defeated both me and the computer.''
\end{quote}

\smallskip
The variable $n$ in \cite{Za96} refers to our $N$, which means there are exactly $6=\binom42$ variables, one per positive root,
exactly as in \S\ref{sec:vandermonde}.
Continuing in the same elementary spirit, the resultant construction of this paper offers one way to work around this particular computational
hurdle.  Indeed, rather than working with all $\binom N2$ pairwise trigonometric differences at once, the resultant considers
the $N-1$ underlying points
$u_1,\dots,u_{N-1}$ directly \eqref{eq:alcove}, from which every pairwise difference is recovered via the Vandermonde determinant
$\Delta$ of Lemma~\ref{lem:sign},  and then eliminates those $N-1$ variables
algebraically via iterated resultants involving $Q(u)=u^m-1$, in place of direct trigonometric simplification.

\smallskip
\textbf{$m=6$:}  Here $K_{\min}(4,6)=3>1$, so this example is the first in this paper where the Binet formula and reduced recursion have real,
multi-term structure rather than a simple geometric sequence which occurs when $K_{\min}=1$.

By Theorem~\ref{thm:resultant}, the resultant has degree $m^{N-1}=6^3=216$. Corollary~\ref{cor:extract}
removes $w^{216-M_4(6)\cdot3!}=w^{216-60}=w^{156}$ and takes a $6$-th root of the rest, leaving the degree-$10$ polynomial
\[
D_4(w) = (w-36)^4(w-108)^4(w-144)^2,
\]
which in expanded form is
\begin{align*}
D_4(w) &= w^{10} - 864\,w^9 + 326592\,w^8 - 70917120\,w^7 \\
&\quad + 9765287424\,w^6 - 888127193088\,w^5 + 53871009251328\,w^4 \\
&\quad - 2147805009543168\,w^3 + 53852167023427584\,w^2 \\
&\quad - 767793272413224960\,w + 4738381338321616896.
\end{align*}
The largest coefficient in the expansion of $D_{4}(w)$ has $19$ digits.
Theorem~\ref{thm:recursion}, applied with no further reduction, gives the unreduced, order $M_4(6)=10$ recursion
$$
\sum_{k=0}^{10}q_kV_{n-k}(4,6)=0
\,\,\,\,\,
\textrm{\rm for $n > 10$}
$$
whose coefficients run from $19$ digits ($q_0=4738381338321616896$) up to $30$ digits ($q_{10}=-q_9=231812806445087701493115518976$).

However, since $D_4(w)$ has only $K_{\min}(4,6)=3$ distinct roots, namely $36,108,144$, of multiplicity $4,4,2$ respectively, we get the
$y$-values from Theorem \ref{thm:binet} to be
$y=c_{4,6}/a=864/a$ equal to $24,8,6$.  As a result, Theorem~\ref{thm:binet} collapses the unreduced order $10$ recursion above to the order $3$ recursion
\[
V_n(4,6) = 38\,V_{n-1}(4,6) - 384\,V_{n-2}(4,6) + 1152\,V_{n-3}(4,6),
\,\,\,\,\,
\textrm{\rm for $n>3$}
\]
with seeds $V_1(4,6)=280$, $V_2(4,6)=5264$, $V_3(4,6)=115552$. Equivalently the closed Binet form
\[
V_n(4,6) = 2\big(4\cdot24^n + 4\cdot8^n + 2\cdot6^n\big)
\,\,\,\,\,
\textrm{\rm for $n\ge 1$. }
\]

\smallskip
\textbf{$m=7$:} $D_4(w)=(w-49)^8(w^3-245w^2+14406w-343^2)^4$.  $K_{\min}=4$. The reduced characteristic polynomial in $y=1372/w$ is
\[
y^4 - 196y^3 + 8624y^2 - 131712y + 614656,
\]
giving the order $4$ recursion
\[
V_n(4,7) = 196\,V_{n-1}(4,7) - 8624\,V_{n-2}(4,7) + 131712\,V_{n-3}(4,7) - 614656\,V_{n-4}(4,7)
\,\,\,\,\,
\textrm{\rm for $n> 4$}
\]
with seeds
\[
V_1(4,7)=1792,\quad V_2(4,7)=175616,\quad V_3(4,7)=23005696,\quad V_4(4,7)=3206045696.
\]
This reproduces $V_5(4,7)=452013105152$ and $V_6(4,7)=63867813330944$. Equivalently, the Binet formula reads
\[
V_n(4,7) = 16\cdot28^n + 8\sum_{k=0}^2\Big(\tfrac{1372}{x_k}\Big)^{\!n},
\]
where $x_0,x_1,x_2$ are the three real roots of $x^3-245x^2+14406x-117649=0$.
\end{example}

\medskip
\begin{example}{$N=5$.}\label{subsec:N5}
Here $c_{5,m}=5m^4$.

\smallskip
\textbf{$m=6$:} This case is small enough to write out every intermediate object in full, and the contrast between unreduced and
reduced recursion formula is, in our opinion, worth seeing explicitly.

By Theorem~\ref{thm:resultant}, the iterated resultant $\widetilde D_5(w)$ has degree $m^{N-1}=6^4=1296$ in $w$. Corollary~\ref{cor:extract}
extracts $D_5(w)$ by removing the factor $w^{1296-M_5(6)\cdot4!}=w^{1296-120}=w^{1176}$ and taking a $4!=24$-th root of what remains of the
original degree $1296$ polynomial. We get
\begin{align*}
D_5(w) &= (w-1296)^5 \\&= w^5 - 6480\,w^4 + 16796160\,w^3 - 21767823360\,w^2 + 14105549537280\,w - 3656158440062976,
\end{align*}
a degree-$5$ polynomial whose coefficients already run to $16$ digits. Theorem~\ref{thm:recursion}, applied directly to this $D_5(w)$ with no further reduction, gives the unreduced, order $M_5(6)=5$ recursion $\sum_{k=0}^5 q_kV_{n-k}(5,6)=0$ for $n>5$, with
\begin{align*}
(q_0,q_1,q_2,q_3,q_4,q_5) &= (-3656158440062976,\ 91403961001574400,\
\\& -914039610015744000,\ 4570198050078720000,\
\\&-11425495125196800000,\ 11425495125196800000),
\end{align*}
each entry already being $16$ to $20$ digits long . These coefficients are exactly $1296^5$ times the coefficients of $(5\tau-1)^5=3125\tau^5-3125\tau^4+1250\tau^3-250\tau^2+25\tau-1$, so their size is partly the overall scale $c_{5,6}^{\,5}=6480^5$
built into the normalization of Theorem~\ref{thm:recursion}.

None of this is necessary because $D_5(w)=(w-1296)^5$ has a single root, so $K_{\min}(5,6)=1$.
Indeed, the bound of Proposition~\ref{prop:bracelet} gives $B(5,6)=1$.  The Binet formula (Theorem~\ref{thm:binet})
 collapses the entire order $5$ recursion above to the order $1$ recursion
\[
V_n(5,6) = 5\,V_{n-1}(5,6)
\,\,\,\,\,
\textrm{\rm for $n > 1$ with}
\,\,\,\,\,
V_1(5,6)=50
\]
which simply becomes the formula that
\[
V_n(5,6)=2\cdot5^{\,n+1}
\,\,\,\,\,
\textrm{for all $n \ge 1$.}
\]

\smallskip
\textbf{$m=7$:} The previous two extended examples reduced to rational roots.  Here $D_5(w)$ has a single irreducible cubic factor, namely
\[
D_5(w) = \big(w^3-2401w^2+1647086w-282475249\big)^5.
\]
The value $K_{\min}(5,7)=3$ arises purely from the Galois-orbit structure of Lemma~\ref{lem:galois}.
By Theorem~\ref{thm:resultant}, the iterated resultant has degree $m^{N-1}=7^4=2401$.  Corollary~\ref{cor:extract} removes $w^{2401-M_5(7)\cdot4!}=w^{2401-360}=w^{2041}$ and takes a $24$-th root of the rest, leaving the degree-$15$ polynomial $D_5(w)$ displayed above, whose full expansion, which is needed for the unreduced recursion, has a constant term of $43$ digits and a leading coefficient, after the substitution $w=c_{5,7}\tau$ of Theorem~\ref{thm:recursion}, of $62$ digits.  Fortunately, none of this is necessary upon obtaining the reduced recursive formula.

With the three roots of the cubic sharing multiplicity $5$ and $y$-value $y=c_{5,7}/w$, the Binet formula collapses the
order $15$ unreduced recursion to the order $3$ recursion, namely
\[
V_n(5,7) = 70\,V_{n-1}(5,7) - 1225\,V_{n-2}(5,7) + 6125\,V_{n-3}(5,7)
\,\,\,\,\,
\textrm{\rm for $n > 3$, }
\]
with seeds $V_1(5,7)=700$, $V_2(5,7)=24500$, $V_3(5,7)=1041250$.
The Binet formula is
\[
V_n(5,7) = 10\left(y_1^n + y_2^n + y_3^n\right)
\,\,\,\,\,
\textrm{\rm for $n \ge 1$, }
\]
where $y_1, y_2, y_3$ are the roots of $y^3 - 70y^2 + 1225y - 6125 = 0$.

\smallskip
\textbf{$m=8$:} Here
\[
D_5(w) = (w-1024)^5(w-512)^{10}(w-256)^{10}\big(w^2-1536w+65536\big)^5,
\]
of degree $5+10+10+2\cdot5=35=M_5(8)$, so $K_{\min}(5,8)=5$. There are three rational roots, together with an irreducible quadratic, whose
degree matches the bound $\varphi(8)/2=2$ of Lemma~\ref{lem:galois} exactly. The roots of the quadratic  are $768\pm512\sqrt2$.

By Theorem~\ref{thm:resultant}, the resultant has degree $m^{N-1}=8^4=4096$. Corollary~\ref{cor:extract} removes
$w^{4096-M_5(8)\cdot4!}=w^{4096-840}=w^{3256}$ and takes a $24$-th root of the rest,  leaving the degree-$35$ polynomial
$D_5(w)$ above.  The constant coefficient of $D_{5}(w)$  has $91$ digits, and is itself the ``small'' output of a construction
whose leading coefficient, $c_{5,8}^{\,35}=20480^{35}$, has $151$ digits.

With $c_{5,8}=5\cdot8^4=20480$, the $y$-values $y_a=c_{5,8}/a$ are $y_{1024}=20$, $y_{512}=40$, $y_{256}=80$,
and $y_{768\mp512\sqrt2}=240\pm160\sqrt2$, which in total are the roots of
\[
y^5 - 620y^4 + 79200y^3 - 3648000y^2 + 66560000y - 409600000 = (y-20)(y-40)(y-80)\big(y^2-480y+6400\big).
\]
By Theorem~\ref{thm:binet}, the Binet formula reads
\[
V_n(5,8) = 2\Big[5\cdot20^n + 10\cdot40^n + 10\cdot80^n + 5\big(240-160\sqrt2\big)^{\!n} + 5\big(240+160\sqrt2\big)^{\!n}\Big]
\,\,\,\,\,
\textrm{\rm for $n \ge 1$.}
\]
By Theorem~\ref{thm:minrec}, the order $5$ recursion
\begin{align*}
V_n(5,8) &= 620\,V_{n-1}(5,8) - 79200\,V_{n-2}(5,8) + 3648000\,V_{n-3}(5,8)
\\&- 66560000\,V_{n-4}(5,8) + 409600000\,V_{n-5}(5,8)
\,\,\,\,\,
\textrm{\rm for $n > 5$,}
\end{align*}
with seeds
\begin{align*}
V_1(5,8)&=7400,\quad V_2(5,8)=2340000,\quad V_3(5,8)=1025360000,\quad \\ V_4(5,8)&=473550400000,\quad V_5(5,8)=220465184000000.
\end{align*}

\smallskip
\textbf{$m=9$:} Here
\begin{align*}
D_5(w) &= (w-729)^{10}\big(w^3-2187w^2+649539w-4782969\big)^5\big(w^3-1458w^2+531441w-43046721\big)^5
\\&\cdot \big(w^3-729w^2+118098w-4782969\big)^{10},
\end{align*}
of degree $10+3\cdot5+3\cdot5+3\cdot10=70=M_5(9)$.  There is one rational root together with three distinct irreducible cubics,
each of degree matching the bound $\varphi(9)/2=3$ of Lemma~\ref{lem:galois}. $K_{\min}(5,9)=1+3+3+3=10$.

By Theorem~\ref{thm:resultant}, the resultant has degree $m^{N-1}=9^4=6561$. Corollary~\ref{cor:extract} removes
$w^{6561-M_5(9)\cdot4!}=w^{6561-1680}=w^{4881}$ and takes a $24$-th root of the rest. Even after this extraction, $D_5(w)$'s own constant term runs to $167$ digits, and the leading coefficient $c_{5,9}^{\,70}=32805^{70}$ has $317$ digits.

With $c_{5,9}=5\cdot9^4=32805$, the ten $y$-values are $y_{729}=45$ together with the roots of
\[
y^3-4455y^2+492075y-7381125=0, \quad y^3-405y^2+36450y-820125=0, \quad y^3-810y^2+164025y-7381125=0,
\]
call them $\{\alpha_1,\alpha_2,\alpha_3\}$, $\{\beta_1,\beta_2,\beta_3\}$, $\{\gamma_1,\gamma_2,\gamma_3\}$ respectively. By Theorem~\ref{thm:binet}, the Binet formula reads
\[
V_n(5,9) = 2\Big[10\cdot45^n + 5\sum_{i=1}^3\alpha_i^{\,n} + 5\sum_{i=1}^3\beta_i^{\,n} + 10\sum_{i=1}^3\gamma_i^{\,n}\Big], \qquad n\ge1,
\]
and by Theorem~\ref{thm:minrec} the same data gives the order $10$ recursion with characteristic polynomial
\begin{align*}
y^{10} &- 5715y^9 + 6688575y^8 - 3353491125y^7 + 880568212500y^6 - 131890370175000y^5
\\&+ 11743683572812500y^4 - 625714356960703125y^3 + 19306706595380859375y^2
\\&- 312768646845169921875y + 2010655586861806640625.
\end{align*}
This yields the recursion formula
\begin{align*}
V_n(5,9) &= 5715\,V_{n-1}(5,9) - 6688575\,V_{n-2}(5,9) + 3353491125\,V_{n-3}(5,9) - 880568212500\,V_{n-4}(5,9)
\\&+ 131890370175000\,V_{n-5}(5,9) - 11743683572812500\,V_{n-6}(5,9) + 625714356960703125\,V_{n-7}(5,9)
\\&- 19306706595380859375\,V_{n-8}(5,9) + 312768646845169921875\,V_{n-9}(5,9)
\\&- 2010655586861806640625\,V_{n-10}(5,9)
\end{align*}
for $n > 10$  with seeds
\begin{align*}
&V_1(5,9)=65700,\ V_2(5,9)=196141500,\ V_3(5,9)=821988506250,\ V_4(5,9)=3556177391587500,
\\&
V_5(5,9)=15434912848584656250,\ V_6(5,9)=67016244398814991406250
\\& V_7(5,9)=290987216868070298404687500,
V_8(5,9)=1263484301865900564451992187500,
\\& V_9(5,9)=5486129342184888764351570039062500,
\\&
V_{10}(5,9)=23821124815261353542519236156347656250.
\end{align*}
This reproduces $V_{11}(5,9)=103432849797672199075728589595006835937500$ exactly, verifying the order $10$ recursion
against an eleventh, independently computed term.

\smallskip
\textbf{$m=10$:} Here $M_5(10) = \binom{9}{4} = 126$, and $c_{5,10} = 5\cdot 10^4 = 50000$. Since $m=10$,
every root of $D_5(w)$ lies in $\mathbb{Q}(\zeta_{10})^+ = \mathbb{Q}(\sqrt5)$.
By Lemma~\ref{lem:galois}, noting that $\varphi(10)/2 = 2$ we have that every irreducible factor over $\mathbb Q$ is linear or quadratic, and every root can be written in $\mathbb Q(\sqrt5)$. Direct computation (verified against 60-digit numerical
evaluation of \eqref{eq:defn}) gives that
\begin{align*}
D_5(w) = \;& (w-125)^5 (w-400)^{10} (w-3125) \\
&\cdot \bigl(w^2 - 450w + 625\bigr)^5 \bigl(w^2-300w+10000\bigr)^{20} \bigl(w^2-700w+10000\bigr)^{10} \\
&\cdot \bigl(w^2-1200w+160000\bigr)^{10} \bigl(w^2-1500w+250000\bigr)^{10},
\end{align*}
a monic polynomial of degree $5+10+1+2(5+20+10+10+10) = 126 = M_5(10)$. So
$K_{\min}(5,10) = 13$, and there are three rational roots ($125,\,400,\,3125$) together with five distinct
irreducible quadratics, each attaining the bound $\varphi(10)/2=2$ of
Lemma~\ref{lem:galois}.

With $y_a := c_{5,10}/a = 50000/a$, the thirteen roots and their $y$-values are
\[
\begin{array}{c|c|c}
a & \mu(a) & y_a \\\hline
125 & 5 & 400 \\
400 & 10 & 125 \\
3125 & 1 & 16 \\
225 \mp 100\sqrt5 & 5 & 18000 \pm 8000\sqrt5 \\
150 \mp 50\sqrt5 & 20 & 750 \pm 250\sqrt5 \\
350 \mp 150\sqrt5 & 10 & 1750 \pm 750\sqrt5 \\
600 \mp 200\sqrt5 & 10 & \tfrac{375}{2} \pm \tfrac{125}{2}\sqrt5 \\
750 \mp 250\sqrt5 & 10 & 150 \pm 50\sqrt5
\end{array}
\]
By Theorem~\ref{thm:binet},
\begin{align}
V_n(5,10) = \;& 2\Bigl[5\cdot 400^n + 10\cdot 125^n + 16^n +5\bigl((18000+8000\sqrt5)^n+(18000-8000\sqrt5)^n\bigr) \notag \\
&+20\bigl((750+250\sqrt5)^n+(750-250\sqrt5)^n\bigr) +10\bigl((1750+750\sqrt5)^n+(1750-750\sqrt5)^n\bigr) \notag \\
&+10\Bigl(\bigl(\tfrac{375+125\sqrt5}{2}\bigr)^n+\bigl(\tfrac{375-125\sqrt5}{2}\bigr)^n\Bigr) +10\bigl((150+50\sqrt5)^n+(150-50\sqrt5)^n\bigr)\Bigr], \qquad n \ge 1. \label{eq:binet-5-10}
\end{align}

Using the Binet formula expression for the $n$th Lucas number, namely that
$$
L_n=\left(\frac{1+\sqrt{5}}{2}\right)^n + \left(\frac{1-\sqrt{5}}{2}\right)^n, \quad n\ge 0,
$$
one easily confirms that \eqref{eq:binet-5-10} equals \eqref{Vn5,10}.

The reduced recursion formula has $13$ terms, so we need
$V_{n}(5,10)$ for $n \leq 13$.  As it turns out, $V_{13}(5,10)$ has more than $60$ digits, and its
coefficient in the reduced recursion formula has more than $30$ digits.  As such, we will not
present the formula in this space.
\end{example}

\medskip
\begin{example}[$N=6$, $m=8$]\label{ex:N6m8-full}
Here $\binom{6}{2}=15$, so $(-1)^{\binom{6}{2}}=-1$ and
$$
W_6(u_1,\dots,u_5;w) = w\prod_{a=1}^5 u_a^5 + \Delta(1,u_1,\dots,u_5)^2.
$$
The alcove $A(8,6)$ has $M_6(8)=\binom{7}{5}=21$ points. At $x=(1,2,3,4,6)$, with $\zeta:=e^{2\pi i/8}$ and $u_a=\zeta^{x_a}$, one computes
$$
\Delta(1,u_1,\dots,u_5)^2 = -8192,
$$
so $r_x=|\Delta|^2=8192$. Indeed
$$
W_6(u_1,\dots,u_5;8192)=8192\cdot\zeta^{5(1+2+3+4+6)}+(-8192)=8192\cdot\zeta^{80}-8192=0.
$$
Since $80\equiv0\pmod8$, this confirms that $w=8192$ is the root predicted by Lemma~\ref{lem:sign}. The $21$ alcove points in this example do not all coincide.  By direct evaluation of $\{r_x\}_{x\in A(8,6)}$ we get that
$$
D_6(w) = (w-8192)^6(w-16384)^3\bigl(w^2-16384w+33554432\bigr)^6,
$$
which is a monic degree-$21$ polynomial, with $6+3+2\cdot 6=21=M_6(8)$.  The irreducible quadratic factor of $D_{6}$
has roots $8192\mp 4096\sqrt2$,  matching the Galois bound $\varphi(8)/2=2$ of Lemma \ref{lem:galois} because
$\mathbb Q(\zeta_8)^+=\mathbb Q(\sqrt2)$.

Note that $(N-1)!=120$, $m^{N-1}=8^5=32768$, so the predicted power of $w$ is $32768-21\cdot 120=30248$.
Direct symbolic computation of the iterated resultant gives
$$
\widetilde D_6(w) = w^{30248}(w-8192)^{720}(w-16384)^{360}\bigl(w^2-16384w+33554432\bigr)^{720},
$$
confirming that  $\widetilde{D}_6(w)=\varepsilon_{6,8}\,w^{30248}D_6(w)^{120}$ with $\varepsilon_{6,8}=1$.
With $c_{6,8}=6\cdot 8^5=196608$, we have $K_{\min}(6,8)=4$ with $y$-values $24,12,48\pm 24\sqrt2$ of
multiplicities $6,3,6,6$ respectively, thus giving the Binet formula
$$
V_n(6,8) = 2\Bigl[\,6\cdot 24^n + 3\cdot 12^n + 6(48+24\sqrt2)^n + 6(48-24\sqrt2)^n\,\Bigr]
\,\,\,\,\,
\textrm{\rm $n\ge 1$.}
$$
Equivalently the reduced order $4$ recursion
$$
V_n(6,8) = 132\,V_{n-1}(6,8) - 4896\,V_{n-2}(6,8) + 69120\,V_{n-3}(6,8) - 331776\,V_{n-4}(6,8)
\,\,\,\,\,
\textrm{\rm for $n > 4$,}
$$
with seeds $V_1(6,8)=1512$, $V_2(6,8)=90720$, $V_3(6,8)=6811776$, $V_4(6,8)=545564160$.  These formulas confirm that $V_5(6,8)=44432934912$.
\end{example}

\medskip
\begin{example}[$N=8$, $m=10$]\label{ex:N8m10}
Here $M_8(10)=\binom{9}{7}=36$, and $c_{8,10}=8\cdot10^7=80000000$. Since $m=10$, every root of $D_8(w)$ again lies in $\mathbb Q(\zeta_{10})^+=\mathbb Q(\sqrt5)$.  By Lemma~\ref{lem:galois}, the factors of $D_{8}$ have degree at most $\varphi(10)/2=2$ over $\mathbb Q$, which
is the same field as the $N=5,m=10$ example
of \S\ref{subsec:N5}.  This is expected since the field depends only on $m$. Despite $M_8(10)=36$ alcove points, direct $80$-digit numerical
evaluation of $\{r_x\}_{x\in A(10,8)}$ collapses to just five distinct values, namely
$$
D_8(w) = \bigl(w^2-3000000w+10^{12}\bigr)^8\bigl(w^2-5000000w+5\cdot10^{12}\bigr)^8(w-4000000)^4,
$$
which is a monic degree-$36$ polynomial, so $K_{\min}(8,10)=5$.  The two irreducible quadratics, with roots $1500000\mp500000\sqrt5$
and $2500000\mp500000\sqrt5$ respectively, both attain the Galois bound $\varphi(10)/2=2$ of Lemma~\ref{lem:galois}.  There is also
the sole rational root, namely $4000000$.

With $y_a:=c_{8,10}/a$, the five $y$-values are $120\pm40\sqrt5$, $40\pm8\sqrt5$, and $20$, of multiplicities $8,8,8,8,4$ respectively.
With this, the Binet formula reads as
\begin{equation}\label{binet Vn8,10}
V_n(8,10) = 2\Bigl[4\cdot20^n + 8(120+40\sqrt5)^n+8(120-40\sqrt5)^n + 8(40+8\sqrt5)^n+8(40-8\sqrt5)^n\Bigr]
\,\,\,\,\,
\textrm{\rm $n \ge 1$.}
\end{equation}

Using the Binet formula expression for the $n$th Lucas number, combined with the Binet formula expression \eqref{eq:Classical_Binet} for the $n$th Fibonacci number,
one easily confirms that \eqref{binet Vn8,10} equals \eqref{Vn8,10}.

Equivalently the reduced order $5$ recursion is
\begin{align*}
V_n(8,10) &= 340\,V_{n-1}(8,10) - 33280\,V_{n-2}(8,10) + 1356800\,V_{n-3}(8,10) \\&- 24576000\,V_{n-4}(8,10) + 163840000\,V_{n-5}(8,10)
\end{align*}
for $n > 5$
with seeds $V_1(8,10)=5280$, $V_2(8,10)=781440$, $V_3(8,10)=150796800$, $V_4(8,10)=30986700800$, $V_5(8,10)=6459253760000$. This reproduces $V_6(8,10)=1351170379776000$, $V_7(8,10)=282898768609280000$, and $V_8(8,10)=59245723978629120000$ exactly.
\end{example}

\section*{Appendix: Proof of Proposition \ref{prop:bracelet}}

\begin{proof}
We compute $B(N,m)$ by applying Burnside's lemma to $X$, the set of $N$-element subsets of
$\mathbb Z/m\mathbb Z$, and $G=D_m$, the dihedral group of order $2m$, splitting $|D_m|=2m$ into the
$m$ rotations and $m$ reflections and counting fixed subsets of each.

Fix $k\in\{0,1,\dots,m-1\}$ and let $\rho_k\colon i\mapsto i+k\pmod m$ be the
corresponding rotation of the $m$ positions $\mathbb Z/m\mathbb Z$. The cyclic group
$\langle\rho_k\rangle$ generated by $\rho_k$ has order equal to the additive order of $k$ in
$\mathbb Z/m\mathbb Z$, namely $m/\gcd(m,k)=m/d_k$.  This is the order of the rotation $\rho_k$.
The orbits of $\langle\rho_k\rangle$ acting on $\mathbb Z/m\mathbb Z$ are the cosets of the
subgroup $\langle k\rangle=\langle d_k\rangle\le\mathbb Z/m\mathbb Z$ generated by $d_k=\gcd(m,k)$. Since
$\langle d_k\rangle$ has order $m/d_k$, each orbit has size $m/d_k$, and since the orbits partition
the $m$ positions, there are exactly $d_k$ of them. Concretely, the orbit of a position $i$ is
$\{i, i+d_k, i+2d_k,\dots\}\pmod m$, so the $d_k$ orbits are simply the residue classes of
$0,1,\dots,d_k-1$ modulo $d_k$.  We refer to these as the cycles of $\rho_k$.

A subset $T\subset\mathbb Z/m\mathbb Z$ is fixed by $\rho_k$  if and only if
$T$ is a union of complete cycles of $\rho_k$.  If $T$ contains one element of a cycle then $T$ must, by
repeated application of $\rho_k$, contain the entire cycle, and conversely any union of cycles is
visibly $\rho_k$-invariant. Thus a $\rho_k$-fixed subset $T$ with $|T|=N$ exists only when $N$ is an
integer multiple of the common cycle length $m/d_k$ (equivalently, only when $m/d_k$ divides
$N$) in which case $T$ consists of exactly $j:=Nd_k/m$ of the $d_k$ available cycles.  There
are $\binom{d_k}{Nd_k/m}$ ways to choose which $j$ cycles to include. At $k=0$ this recovers the
identity rotation.  In this case, $d_0=\gcd(m,0)=m$, so there are $m$ cycles of length $1$, the condition
$m/d_0=1\mid N$ holds automatically, and the count $\binom{m}{N}$ correctly reproduces every
$N$-subset, as it must since the identity fixes all of them.  Continuing in general, by summing
$\binom{d_k}{Nd_k/m}\mathbf 1_{m/d_k\mid N}$ over the $m$ rotations $k=0,1,\dots,m-1$ gives the total
number of rotation-fixed $N$-subsets counted in Burnside's lemma.  This is the first term
on the right-hand-side of \eqref{eq:Burnside_count}.  Note that the factor $1/2m$ is present because
$|D_m|=2m$.  It remains to compute the reflection contribution $R(N,m)$.

Every reflection in $D_m$ has the form $\sigma_c\colon i\mapsto c-i\pmod m$ for
some $c\in\mathbb Z/m\mathbb Z$, and every choice of $c$ gives a reflection.  So when $c$ ranges over
$\mathbb Z/m\mathbb Z$ we obtain each of the $m$ reflections exactly once. Each $\sigma_c$ is an
involution, since $\sigma_c(\sigma_c(i))=c-(c-i)=i$, and its fixed points are the solutions of
$2i\equiv c\pmod m$.  Let us consider the cases when $m$ is odd and when $m$ is even separately.

Assume $m$ is odd.   Since $\gcd(2,m)=1$, the congruence $2i\equiv c\pmod m$ has the unique
solution $i\equiv c\cdot 2^{-1}\pmod m$ for every $c$.  Each of the $m$ reflections fixes precisely one
position and pairs the remaining $m-1$ positions into $2$-element orbits
$\{i,c-i\}$ with $i\ne c-i$, of which there are $(m-1)/2$. A subset $T$ of size $N$ is fixed by
$\sigma_c$ exactly when $T$ is a union of complete orbits.  The fixed point of the reflection may or may not belong to
$T$, and each swapped pair contributes either both its elements or neither. If the fixed point is
excluded, $T$ is a union of $N/2$ pairs, so then $N$ even; if the fixed point is included, $T$ is the
axis point together with $(N-1)/2$ pairs, so then $N$ odd. Exactly one of these two cases is
arithmetically possible for a given $N$, and in either case the number of pairs required is
$\lfloor N/2\rfloor$, hence the number of $\sigma_c$-fixed $N$-subsets is $\binom{(m-1)/2}{\lfloor
N/2\rfloor}$ regardless of the parity of $N$. Summing over the $m$ reflections gives the term
$m\binom{(m-1)/2}{\lfloor N/2\rfloor}$.

Assume $m$ is even.   Now $\gcd(2,m)=2$, so $2i\equiv c\pmod m$ is solvable only when $c$ is even, in
which case it has exactly $2$ solutions modulo $m$, namely $i$ and $i+m/2$. Thus the $m$ reflections
split into two types according to the parity of $c$, with $m/2$ reflections of each type.

If $c$ is even, then $\sigma_c$ fixes the two positions $i,i+m/2$
and pairs the remaining $m-2$ positions into $(m-2)/2$ orbits of size $2$. A fixed $N$-subset
chooses some $f\in\{0,1,2\}$ of the two fixed points together with $(N-f)/2$ complete pairs, which
requires $N\equiv f\pmod 2$.  By summing the valid choices of $f$ gives $\sum_{f=0}^{2}\binom{2}{f}
\binom{(m-2)/2}{(N-f)/2}$ fixed $N$-subsets for each such reflection, where a term is omitted
whenever $(N-f)/2$ fails to be a nonnegative integer.

If $c$ is odd, then  $\sigma_c$ has no fixed points, thus pairing all $m$ positions into $m/2$
orbits of size $2$. A fixed $N$-subset is then a union of exactly $N/2$ complete pairs, which is only
possible when $N$ is even, giving $\binom{m/2}{N/2}\mathbf 1_{2\mid N}$ fixed $N$-subsets for each
such reflection.

By multiplying each count by the $m/2$ reflections of the corresponding type and adding the two
contributions gives
$$
R(N,m) = \frac{m}{2}\sum_{f=0}^{2}\binom{2}{f}\binom{(m-2)/2}{(N-f)/2} + \frac{m}{2}\binom{m/2}{N/2}\mathbf 1_{2\mid N},
$$
which is why $R(N,m)$ splits into two separate pieces when $m$ is even, unlike the single unified
term available when $m$ is odd.

Adding the rotation and reflection contributions and dividing by $2m$ gives $B(N,m)$, by Burnside's
lemma.
\end{proof}

\vspace{4mm}
\noindent
Jay Jorgenson \\
 Department of Mathematics \\
 The City College of New York \\
 Convent Avenue at 138th Street \\
 New York, NY 10031 U.S.A. \\
 e-mail: jjorgenson@mindspring.com

\vspace{4mm}
\noindent
Anders Karlsson \\
 Section de mathématiques\\
 Université de Genève\\
 2-4 Rue du Liévre\\
 Case Postale 64, 1211\\
 Genève 4, Suisse\\
 e-mail: anders.karlsson@unige.ch \\
 and \\
 Matematiska institutionen \\
 Uppsala universitet \\
Box 256, 751 05 \\
 Uppsala, Sweden \\
 e-mail: anders.karlsson@math.uu.se

\vspace{4mm}
\noindent
Lejla Smajlovi\'{c} \\
 Department of Mathematics and Computer Science\\
 University of Sarajevo\\
 Zmaja od Bosne 35, 71 000 Sarajevo\\
 Bosnia and Herzegovina\\
 e-mail: lejlas@pmf.unsa.ba

\end{document}